\documentclass[11pt]{amsart}

\usepackage{a4wide}
\usepackage{amssymb}
\usepackage{mathrsfs,mathtools}
\usepackage{enumerate}
\usepackage{esint}
\usepackage{titletoc}
\usepackage{hyperref}

\theoremstyle{plain}
\newtheorem{theorem}{Theorem}[section]
\newtheorem{lemma}[theorem]{Lemma}
\newtheorem{corollary}[theorem]{Corollary}
\newtheorem{proposition}[theorem]{Proposition}

\theoremstyle{definition}
\newtheorem{definition}[theorem]{Definition}
\newtheorem{remark}[theorem]{Remark}

\numberwithin{equation}{section}

\DeclareMathOperator{\supp}{supp}
\DeclareMathOperator*{\essliminf}{ess\,lim\,inf}

\makeatletter
\@namedef{subjclassname@2020}{\textup{2020} Mathematics Subject Classification}
\makeatother

\title[Liouville theorem for singular solutions to nonlocal equations]{Liouville theorem for singular solutions to nonlocal equations}

\author{Minhyun Kim}
\address{Department of Mathematics \& Research Institute for Natural Sciences, Hanyang University, 04763 Seoul, Republic of Korea}
\email{minhyun@hanyang.ac.kr}

\author{Se-Chan Lee}
\address{School of Mathematics, Korea Institute for Advanced Study, 02455 Seoul, Republic of Korea}
\email{sechan@kias.re.kr}

\thanks{M. Kim is supported by the National Research Foundation of Korea (NRF) grant funded by the Korean government (MSIT) (RS-2023-00252297). S.-C. Lee is supported by the KIAS Individual Grant (No. MG099001) at the Korea Institute for Advanced Study.}

\begin{document}

	\begin{abstract}
        We study singular solutions to the fractional Laplace equation and, more generally, to nonlocal linear equations with measurable kernels. We establish B\^ocher type results that characterize the behavior of singular solutions near the singular point. In addition, we prove Liouville theorems for singular solutions. To this end, we construct fundamental solutions for nonlocal linear operators and establish a localized comparison principle.
	\end{abstract}
	
	\maketitle

\section{Introduction}

The analysis of singularities in solutions to partial differential equations (PDEs) is a central topic in regularity theory. A fundamental question in this area is whether solutions to a given PDE are always smooth or can develop singularities. This question arises in a wide range of mathematical and physical contexts, including Hilbert’s nineteenth problem, singular harmonic maps, the Navier--Stokes equations in fluid dynamics, and the Poincar\'e conjecture via Ricci flow.

For the Laplace equation, one of the most classical and fundamental PDEs, the classification and asymptotic behavior of singular solutions are now well understood. For example, the early work of B\^ocher~\cite{Boc03} in 1903 established that any nonnegative singular harmonic function in a punctured ball (in dimension greater than or equal to two) can be represented as a linear combination of the fundamental solution and a harmonic function in the ball. By applying the Kelvin transform, he also showed that any nonnegative harmonic function in an exterior domain approaches a finite limit at infinity, provided that the dimension is greater than two. As a consequence of B\^ocher's result, the following Liouville theorem for singular harmonic solutions holds. See also Picard~\cite{Pic24}, Kellogg~\cite{Kel26}, and Raynor~\cite{Ray26}.

\begin{theorem}\cite{Boc03}\label{thm-sing-Liouv-Lap}
Assume that $n>2$. Let $u$ be harmonic in $\mathbb{R}^n \setminus \{0\}$. Suppose that $u$ is bounded from above or below in a neighborhood of $0$ and that $u$ is bounded from above or below in a neighborhood of infinity. Then
\begin{equation*}
u=a|x|^{2-n}+b
\end{equation*}
for some constants $a, b \in \mathbb{R}$.
\end{theorem}

An extension of Theorem~\ref{thm-sing-Liouv-Lap} to certain linear equations can also be found in \cite{Boc03}. Gilbarg--Serrin~\cite{GS56} carried out a comprehensive study of fundamental solutions and isolated singularities for second-order linear elliptic equations. Building on this, Serrin~\cite{Ser64a,Ser65a} published a series of influential papers on singular solutions of general quasilinear divergence form equations, including the $p$-Laplace equation as a model case. For further developments on quasilinear equations, we refer the reader to Serrin~\cite{Ser65b}, Kichenassamy--V\'eron~\cite{KV86}, V\'eron~\cite{Ver96}, and Nicolosi--Skrypnik--Skrypnik~\cite{NSS03}. There are also results on equations in non-divergence form; see Labutin~\cite{Lab00,Lab01} for Pucci’s extremal equations and Armstrong--Sirakov--Smart~\cite{ASS11} for fully nonlinear equations. In particular, the results in \cite{ASS11} extend Theorem~\ref{thm-sing-Liouv-Lap} to the cases $n=1$ and $n=2$ with the corresponding fundamental solutions. These classical results provide a foundational framework for the characterization of singularities in local elliptic equations.

On the other hand, nonlocal equations such as the fractional Laplace equation have attracted considerable attention in recent decades as they arise naturally in modeling various phenomena described by nonlocal interactions. For instance, the fractional Laplacian is the infinitesimal generator of a stable L\'evy jump process. Moreover, it plays a crucial role in the study of the Boltzmann equation as it can approximate the Boltzmann collision operator; see Silvestre~\cite{Sil16}. Furthermore, there are a number of significant applications in materials science \cite{Bat06,BC99}, soft thin films \cite{Kur06}, quasi-geostrophic dynamics \cite{CV10}, and image processing \cite{GO08}.

While the regularity theory for nonlocal equations is now well established---thanks to the nonlocal extensions of the classical De Giorgi--Nash--Moser theory for both elliptic \cite{Coz17,DCKP14,DCKP16,Kas09} and parabolic \cite{CCV11,FK13,KW24b} problems---the study of singular solutions is still in its early stage. The aim is to extend the classical program of analyzing singular solutions to the fractional Laplace equation and, more generally, to nonlocal linear equations. In this direction, B\^ocher type results characterizing nonnegative singular solutions near the origin have been obtained; see Li--Wu--Xu~\cite{LWX18}, Li--Liu--Wu--Xu \cite{LLWX20}, and Klimsiak~\cite{Kli25}. Moreover, the isolated singularity theorem and removable singularity theorem for nonlocal nonlinear equations, extending Serrin's works~\cite{Ser64a,Ser65a,Ser65b}, have been studied by the authors in \cite{KL24} and by the first author in \cite{Kim25}.

In this work, we take up the study of singular solutions, focusing on Liouville type theorems. Our goal is to establish the nonlocal counterpart of Theorem~\ref{thm-sing-Liouv-Lap} for nonlocal linear equations with measurable kernels. We begin with the simplest case, where the operator is given by the fractional Laplacian. The fractional Laplacian of order $s \in (0,1)$ is defined by
\begin{equation*}
(-\Delta)^s u(x) = \frac{4^{s}\Gamma(\frac{n}{2}+s)}{\pi^{n/2}|\Gamma(-s)|} \,\mathrm{p.v.} \int_{\mathbb{R}^n} \frac{u(x)-u(y)}{|x-y|^{n+2s}} \,\mathrm{d}y,
\end{equation*}
where $\Gamma$ is the gamma function. Throughout the paper, we always assume that $n \geq 2$. It is well known that the function
\begin{equation}\label{eq-FS}
\Phi_s(x) \coloneqq \frac{\Gamma(\frac{n}{2}-s)}{4^{s}\pi^{n/2}\Gamma(s)} |x|^{2s-n}
\end{equation}
is the fundamental solution for the fractional Laplacian, that is, $(-\Delta)^s\Phi_s = \delta_0$ in $\mathscr{D}'(\mathbb{R}^n)$ (in the sense of Definition~\ref{def-distribution}), where $\delta_0$ is the Dirac delta at the origin; see, for instance, classical works of Riesz~\cite{Rie49} and Blumenthal--Getoor--Ray~\cite{BGR61}; the standard monograph by Samko--Kilbas--Marichev~\cite{SKM93}; and more recent expositions such as Bucur~\cite[Theorem~2.3]{Buc16}, Garofalo~\cite[Theorem~8.4]{Gar19}, and Fern\'andez-Real--Ros-Oton~\cite[Lemma~1.10.9]{FRRO24}.

Our first result is the following Liouville theorem for singular solutions to the fractional Laplace equation, which extends both the classical Liouville theorem (see Theorem~\ref{thm-Liouville-s}) for globally nonnegative solutions and B\^ocher's Theorem~\ref{thm-sing-Liouv-Lap}. We refer the reader to Definition~\ref{def-harmonic} for the precise definition of a $(-\Delta)^s$-harmonic function.

\begin{theorem}\label{thm-sing-Liouv-s}
Let $u$ be $(-\Delta)^s$-harmonic in $\mathbb{R}^n \setminus \{0\}$. Suppose that $u$ is bounded from above or below in a neighborhood of $0$ and that $u$ is bounded from above or below in a neighborhood of infinity. Then
\begin{equation*}
u=a|x|^{2s-n}+b
\end{equation*}
for some constants $a, b \in \mathbb{R}$.
\end{theorem}

The main ingredients in the proof of Theorem~\ref{thm-sing-Liouv-s} are a nonlocal Caccioppoli--Cimmino--Weyl lemma (Theorem~\ref{thm-Weyl}), the classical Liouville Theorem~\ref{thm-Liouville-s} for the fractional Laplacian, and a weak B\^ocher type Theorem~\ref{thm-weak-Bocher}, which we establish below in a more general setting.

The next goal of this paper is to prove a Liouville theorem for more general nonlocal linear operators of the form
\begin{equation}\label{eq-L}
\mathcal{L}u(x) = 2\,\mathrm{p.v.} \int_{\mathbb{R}^n} (u(x)-u(y)) k(x, y)\,\mathrm{d}y,
\end{equation}
where $k: \mathbb{R}^n \times \mathbb{R}^n \to [0, \infty]$ is a measurable kernel satisfying the ellipticity and symmetry conditions
\begin{equation*}
\frac{\Lambda^{-1}}{|x-y|^{n+2s}} \leq k(x, y) = k(y, x) \leq \frac{\Lambda}{|x-y|^{n+2s}}, \quad x, y \in \mathbb{R}^n,
\end{equation*}
with $\Lambda \geq 1$. Note that $\mathcal{L}=(-\Delta)^s$ when the kernel $k$ is given by
\begin{equation*}
k(x, y) = \frac{2^{2s-1}\Gamma(\frac{n}{2}+s)}{\pi^{n/2}|\Gamma(-s)|}|x-y|^{-n-2s}.
\end{equation*}
Although Theorem~\ref{thm-sing-Liouv-L}, a Liouville theorem for singular solutions of $\mathcal{L}u=0$, can be regarded as an extension of Theorem~\ref{thm-sing-Liouv-s}, its proof is significantly more challenging. One of the main reasons is the absence of a nonlocal Caccioppoli--Cimmino--Weyl lemma for $\mathcal{L}$; as in the local case\footnote{The classical Caccioppoli--Cimmino--Weyl lemma~\cite{Cac38,Cim37,Cim38,Wey40} states that if $u \in \mathscr{D}'(\Omega)$ satisfies $\Delta u=0$ in $\mathscr{D}'(\Omega)$, then $u \in C^\infty(\Omega)$ and $\Delta u=0$ pointwise in $\Omega$. However, such a hypoellipticity fails in general for second-order elliptic equations with bounded measurable coefficients; see Serrin~\cite{Ser64b}.}, such a result is not expected to hold for general nonlocal operators $\mathcal{L}$ with measurable kernels $k$.

The first step in proving the Liouville theorem is to find the fundamental solution for $\mathcal{L}$ and to investigate its asymptotic behavior both near the origin and at infinity. While the fundamental solution for the fractional Laplacian is defined as the distributional solution of $(-\Delta)^s \Phi_s = \delta_0$ in $\mathscr{D}'(\mathbb{R}^n)$, this definition is not appropriate for operators $\mathcal{L}$ with measurable kernels $k$, because $\mathcal{L}\varphi$ is not well defined in general even for nice test functions $\varphi \in C^\infty_c(\mathbb{R}^n)$. However, the following theorem provides a precise framework for interpreting $\mathcal{L}\Phi=\delta_0$ in this more general situation. See Definition~\ref{def-harmonic} for the precise definition of an $\mathcal{L}$-harmonic function.

\begin{theorem}\label{thm-weak-Bocher}
Let $\Omega \subset \mathbb{R}^n$ be an open set containing the origin. Let $u$ be $\mathcal{L}$-harmonic in $\Omega \setminus \{0\}$ and bounded from below in $\Omega$. Then there exists a constant $a \geq 0$ such that
\begin{equation}\label{eq-delta}
	\mathcal{L}u=a\delta_0 \quad\text{in }\Omega
\end{equation}
in the sense that
\begin{equation}\label{eq-delta-weak}
	\int_{\mathbb{R}^n} \int_{\mathbb{R}^n} (u(x)-u(y))(\varphi(x)-\varphi(y)) k(x, y) \,\mathrm{d}y\,\mathrm{d}x=a\varphi(0)
\end{equation}
for all $\varphi \in C_c^{\infty}(\Omega)$.
\end{theorem}

B\^ocher type results similar to Theorem~\ref{thm-weak-Bocher} have been obtained for certain classes of nonlocal operators. For example, Li--Wu--Xu~\cite{LWX18} considered the fractional Laplacian, Li--Liu--Wu--Xu \cite{LLWX20} extended this to the fractional Laplacian with a drift term, and Klimsiak~\cite{Kli25} generalized these results to general L\'evy operators with drift. These results differ from Theorem~\ref{thm-weak-Bocher} in two main respects. First, they interpret the equation \eqref{eq-delta} in the distributional sense, which is feasible for L\'evy operators, whereas here it is interpreted in the sense of \eqref{eq-delta-weak}. Second, they assume that $u$ is globally nonnegative in $\mathbb{R}^n$, which is not necessary in Theorem~\ref{thm-weak-Bocher}.

With the aid of Theorem~\ref{thm-weak-Bocher}, we establish the existence, uniqueness, and asymptotic behavior of the fundamental solution for the operator $\mathcal{L}$. Here and after, we write $A \lesssim B$ (resp.\ $A \gtrsim B$) if there is a constant $C>0$ such that $A \leq CB$ (resp.\ $CA \geq B$), and $A \eqsim B$ if $A \lesssim B$ and $A \gtrsim B$. See Definition~\ref{def-superharmonic} for the definition of an $\mathcal{L}$-superharmonic function.

\begin{theorem}[Fundamental solution]\label{thm-FS}
There exists a unique function $\Phi: \mathbb{R}^n \to [0, \infty]$ that satisfies the following properties\/\textup{:}
\begin{enumerate}[(i)]
\item\label{eq-FS1}
$\mathcal{L}\Phi = \delta_0$ in $\mathbb{R}^n$ in the sense that
\begin{equation*}
\int_{\mathbb{R}^n} \int_{\mathbb{R}^n} (\Phi(x)-\Phi(y))(\varphi(x)-\varphi(y)) k(x, y) \,\mathrm{d}y \,\mathrm{d}x= \varphi(0)
\end{equation*}
for all $\varphi \in C^\infty_c(\mathbb{R}^n)$\textup{;}
\item\label{eq-FS2}
$\Phi$ is $\mathcal{L}$-harmonic in $\mathbb{R}^n \setminus \{0\}$ and $\mathcal{L}$-superharmonic in $\mathbb{R}^n$\textup{;}
\item\label{eq-FS3}
$\lim_{x \to 0} \Phi(x) = \infty$ and $\lim_{x \to \infty} \Phi(x) = 0$.
\end{enumerate}
Moreover, it satisfies
\begin{align}\label{eq-FS-bounds}
\begin{split}
\Phi(x) &\lesssim |x|^{2s-n} \quad\text{in }\mathbb{R}^n\text{ and} \\
\Phi(x) &\gtrsim |x|^{2s-n} \quad\text{in a neighborhood of the origin}.
\end{split}
\end{align}
The function $\Phi$ is called the \emph{fundamental solution} for the operator $\mathcal{L}$.
\end{theorem}

We are now ready to state the Liouville theorem for singular solutions to the equation $\mathcal{L}u=0$.

\begin{theorem}\label{thm-sing-Liouv-L}
Let $u$ be $\mathcal{L}$-harmonic in $\mathbb{R}^n \setminus \{0\}$. Suppose that $u$ is bounded from above or below in a neighborhood of $0$ and that $u$ is bounded from above or below in a neighborhood of infinity. Then
\begin{equation*}
u=a\Phi+b
\end{equation*}
for some constants $a, b \in \mathbb{R}$.
\end{theorem}

The proof of Theorem~\ref{thm-sing-Liouv-L} consists of two main parts. The first is the classification of singular solutions near the origin, which is given by Theorem~\ref{thm-Bocher} below. The second deals with the behavior of singular solutions at infinity; see Lemmas~\ref{lem-bdd} and \ref{lem-bdd2} for details.

\begin{theorem}\label{thm-Bocher}
Let $\Omega \subset \mathbb{R}^n$ be an open set containing the origin. Let $u$ be $\mathcal{L}$-harmonic in $\Omega \setminus \{0\}$ and bounded from below in $\Omega$. Then there exist a constant $a \geq 0$ and an $\mathcal{L}$-harmonic function $v$ in $\Omega$ such that
\begin{equation*}
u = a\Phi + v.
\end{equation*}
In particular, if $a>0$, then
\begin{equation*}
u(x) \eqsim |x|^{2s-n}
\end{equation*}
in a neighborhood of the origin.
\end{theorem}

One might expect that Theorem~\ref{thm-Bocher} would follow immediately from Theorem~\ref{thm-weak-Bocher} by subtracting equation \eqref{eq-delta} from the corresponding equation satisfied by $a\Phi$. However, this is not the case, as one must additionally show that $u-a\Phi$ is $\mathcal{L}$-harmonic in $\Omega$, which requires regularity of this function near the origin. The proof of this fact is rather involved and relies on nontrivial tools, including the isolated singularity theorem established in our earlier paper~\cite{KL24} (see Theorem~\ref{thm-iso-sing}) and localized maximum/comparison principles. To this end, we develop the nonlocal counterpart of the approach by Gilbarg--Serrin~\cite{GS56}.

Several maximum and comparison principles for nonlocal operators are available in the literature, but they typically require pointwise comparison of solutions in the complement of the domain as an assumption; see, for example, Silvestre~\cite[Proposition~2.21]{Sil07}, Lindgren--Lindqvist~\cite[Lemma~9]{LL14}, Korvenp\"a\"a--Kuusi--Palatucci~\cite[Lemma~6]{KKP17}, Fern\'andez-Real--Ros-Oton~\cite[Lemma~2.3.3 and Corollary 2.3.4]{FRRO24}, Bj\"orn--Bj\"orn--Kim~\cite[Corollary~3.6]{BBK24}, and Kim--Lee~\cite[Lemma~5.2]{KL23}. This pointwise condition is quite natural due to the nonlocal nature of the operators. However, they are not applicable in our setting, as we aim to compare singular solutions with the fundamental solution without knowing a priori which one dominates the other outside the domain. Therefore, we need a new comparison principle that does not require pointwise comparison in the complement of the domain. For this purpose, we establish the localized maximum and comparison principles in Theorem~\ref{thm-loc-max} and Corollary~\ref{cor-loc-comp}; they encode the relevant information about functions outside the domain in integral quantities rather than in a pointwise comparison.

For the precise definitions of a weak super- and subsolution of $\mathcal{L}u=0$ and the tail term, we refer the reader to Definition~\ref{def-harmonic} and \eqref{eq-tail}, respectively.

\begin{theorem}[Localized maximum principle]\label{thm-loc-max}
Let $0<r<R$ and $\theta \geq 2$. Let $\Omega' \subset \mathbb{R}^n$ be an open set such that $B_R \setminus \overline{B}_r \Subset \Omega'$. Let $u$ be a weak supersolution of $\mathcal{L}u=0$ in $B_R \setminus \overline{B}_r$ such that $u \in W^{s, 2}(\Omega')$ and $u \geq 0$ in $\overline{B}_{\theta R} \setminus (B_R \setminus \overline{B}_r)$. Then there exists a constant $C=C(n, s, \Lambda)>0$ such that
\begin{equation*}
u \geq - C\theta^{-2s}\mathrm{Tail}(u_-; 0, \theta R)
\end{equation*}
a.e.\ in $B_R \setminus \overline{B}_r$.
\end{theorem}

\begin{corollary}[Localized comparison principle]\label{cor-loc-comp}
Let $0<r<R$ and $\theta \geq 2$. Let $\Omega' \subset \mathbb{R}^n$ be an open set such that $B_R \setminus \overline{B}_r \Subset \Omega'$. Let $u, v \in W^{s, 2}(\Omega')$ be a weak super- and subsolution of $\mathcal{L}u=0$ in $B_R \setminus \overline{B}_r$, respectively, and suppose that $u \geq v$ in $\overline{B}_{\theta R} \setminus (B_R \setminus \overline{B}_r)$. Then there exists a constant $C=C(n, s, \Lambda)>0$ such that
\begin{equation*}
u \geq v - C \theta^{-2s} \mathrm{Tail}((u-v)_-; 0, \theta R)
\end{equation*}
a.e.\ in $B_R \setminus \overline{B}_r$.
\end{corollary}

We point out that we restrict our attention to the case $n \geq 2$, as Lemma~\ref{lem-harnack} requires this dimensional condition. Apart from this lemma, all other arguments work for $n>2s$. The critical case $n=2s$ remains largely open.

The paper is organized as follows. In Section~\ref{sec-preliminaries}, we collect several definitions and well-known results concerning function spaces and the notion of solutions relevant to our setting. Section~\ref{sec-fL} is devoted to the proof of the Liouville Theorem~\ref{thm-sing-Liouv-s} for singular solutions to the fractional Laplace equation. In Section~\ref{sec-weak-Bocher}, we establish a weak Bôcher type theorem (Theorem~\ref{thm-weak-Bocher}) for general operators $\mathcal{L}$. Section~\ref{sec-fundamentalsol} is concerned with the construction of the fundamental solution for $\mathcal{L}$ and its properties in Theorem~\ref{thm-FS}. In Section~\ref{sec-Bocher}, we prove the B\^ocher type Theorem~\ref{thm-Bocher} near the origin with the help of the isolated singularity theorem. For this purpose, we also provide the proofs of Theorem~\ref{thm-loc-max} and Corollary~\ref{cor-loc-comp}. Finally, in Section~\ref{sec-liouville}, the proof of the Liouville Theorem~\ref{thm-sing-Liouv-L} for singular solutions to the equation $\mathcal{L}u=0$ is provided.

\section{Preliminaries}\label{sec-preliminaries}

In this section, we recall the definitions of several function spaces and the notion of solutions associated with the fractional Laplacian and the nonlocal operators in \eqref{eq-L}. We also summarize some well-known results that will be used throughout the paper. Throughout this section, we assume that $\Omega \subset \mathbb{R}^n$ is a nonempty open set, which may be unbounded unless otherwise stated.

Let $\mathscr{S}(\mathbb{R}^n)$ be the Schwartz space of rapidly decreasing functions on $\mathbb{R}^n$, and let $\mathscr{S}'(\mathbb{R}^n)$ be its topological dual, the space of tempered distributions. We also denote by $\mathscr{E}'(\mathbb{R}^n)$ the dual of $C^\infty(\mathbb{R}^n)$, that is, the space of compactly supported distributions, and by $\mathscr{D}'(\mathbb{R}^n)$ the dual of $C^\infty_c(\mathbb{R}^n)$, the space of all distributions on $\mathbb{R}^n$. Following Silvestre~\cite[Definition~2.3]{Sil07}, we denote by $\mathscr{S}_s(\mathbb{R}^n)$ the space $C^\infty(\mathbb{R}^n)$ endowed with the countable family of seminorms
\begin{equation*}
    [u]_{\mathscr{S}_s(\mathbb{R}^n)}^{\alpha} \coloneqq \sup_{x \in \mathbb{R}^n} (1+|x|^{n+2s}) |\partial^\alpha u(x)|,
\end{equation*}
and by $\mathscr{S}_s'(\mathbb{R}^n)$ its topological dual. Note that
\begin{align*}
&C^\infty_c(\mathbb{R}^n) \hookrightarrow \mathscr{S}(\mathbb{R}^n) \hookrightarrow \mathscr{S}_s(\mathbb{R}^n) \hookrightarrow C^\infty(\mathbb{R}^n) \quad\text{and}\\
&\mathscr{E}'(\mathbb{R}^n) \hookrightarrow \mathscr{S}_s'(\mathbb{R}^n) \hookrightarrow \mathscr{S'}(\mathbb{R}^n) \hookrightarrow \mathscr{D}'(\mathbb{R}^n).
\end{align*}
Moreover, if $\varphi \in \mathscr{S}(\mathbb{R}^n)$, then $(-\Delta)^s \varphi \in \mathscr{S}_s(\mathbb{R}^n)$ by Lemma~8.1 in Garofalo~\cite{Gar19}.

We now recall the definition of a distributional solution with respect to the fractional Laplacian; see \cite{Sil07}, \cite[Definition~1.5]{Buc16}, or \cite[Definition~8.2]{Gar19}.

\begin{definition}\label{def-distribution}
    Let $T \in \mathscr{S}'(\mathbb{R}^n)$. We say that $u \in \mathscr{S}_s'(\mathbb{R}^n)$ solves $(-\Delta)^su=T$ in $\mathscr{D}'(\mathbb{R}^n)$ (or is a \emph{distributional solution} of $(-\Delta)^s u = T$ in $\mathbb{R}^n$) if
    \begin{equation*}
        \langle u, (-\Delta)^s\varphi \rangle = \langle T, \varphi \rangle
    \end{equation*}
    for all $\varphi \in \mathscr{S}(\mathbb{R}^n)$.
\end{definition}

We note that Definition~\ref{def-distribution} cannot be adapted to the operators $\mathcal{L}$ in \eqref{eq-L} with measurable kernels $k$, since $\mathcal{L}\varphi$ is not well defined in general even for $\varphi \in C^\infty_c(\mathbb{R}^n)$.

In what follows, we also present other notions of solutions, namely weak solutions and $\mathcal{L}$-harmonic functions. For this purpose, we recall several function spaces. The fractional Sobolev space $W^{s, 2}(\Omega)$ consists of measurable functions $u: \Omega \to [-\infty, \infty]$ whose fractional Sobolev norm
\begin{align*}
\|u\|_{W^{s, 2}(\Omega)}
&\coloneqq \left( \|u\|_{L^2(\Omega)}^2 + [u]_{W^{s, 2}(\Omega)}^2 \right)^{1/2} \\
&\coloneqq \left( \int_\Omega |u(x)|^2 \,\mathrm{d}x + \int_\Omega \int_\Omega \frac{|u(x)-u(y)|^2}{|x-y|^{n+2s}} \,\mathrm{d}y\,\mathrm{d}x \right)^{1/2}
\end{align*}
is finite. By $W^{s, 2}_{\mathrm{loc}}(\Omega)$ we denote the space of functions $u$ such that $u \in W^{s, 2}(G)$ for every open $G \Subset \Omega$. We refer the reader to Di Nezza--Palatucci--Valdinoci~\cite{DNPV12} for properties of these spaces.

Since we are concerned with nonlocal equations, we also need the \emph{tail space}
\begin{equation*}
L^{1}_{2s}(\mathbb{R}^n) \coloneqq \left\{ u~\text{measurable}: \int_{\mathbb{R}^n} \frac{|u(y)|}{(1+|y|)^{n+2s}} \,\mathrm{d}y < \infty \right\};
\end{equation*}
see Kassmann~\cite{Kas11} and Di Castro--Kuusi--Palatucci~\cite{DCKP16}. Note that $u \in L^{1}_{2s}(\mathbb{R}^n)$ if and only if the \emph{nonlocal tail} (or \emph{tail} for short)
\begin{equation}\label{eq-tail}
\mathrm{Tail}(u; x_0, r) \coloneqq r^{2s} \int_{\mathbb{R}^n \setminus B_r(x_0)} \frac{|u(y)|}{|y-x_0|^{n+2s}} \,\mathrm{d}y
\end{equation}
is finite for any $x_0 \in \mathbb{R}^n$ and $r>0$.

Next, we recall the definitions of a weak solution to the equation $\mathcal{L}u=0$, or more generally $\mathcal{L}u=f$, and of an $\mathcal{L}$-harmonic function. To this end, we define, for measurable functions $u, v: \mathbb{R}^n \to [-\infty, \infty]$, the bilinear form
\begin{equation*}
\mathcal{E}(u,v)\coloneqq\int_{\mathbb{R}^n} \int_{\mathbb{R}^n} (u(x)-u(y))(v(x)-v(y)) k(x, y) \,\mathrm{d}y\,\mathrm{d}x,
\end{equation*}
provided that the integral is finite.

\begin{definition}\label{def-harmonic}
Let $f \in L^\infty(\Omega)$. A function $u \in W^{s, 2}_{\mathrm{loc}}(\Omega) \cap L^{1}_{2s}(\mathbb{R}^n)$ is a {\it weak supersolution} (resp.\ \emph{weak subsolution}) of $\mathcal{L}u=f$ in $\Omega$ if
\begin{equation}\label{eq-sol}
\mathcal{E}(u, \varphi) \geq \int_\Omega f\varphi \,\mathrm{d}x ~\left(\text{resp.} \leq  \int_\Omega f\varphi \,\mathrm{d}x\right)
\end{equation}
for all nonnegative $\varphi \in C^\infty_c(\Omega)$. A function $u \in W^{s, 2}_{\mathrm{loc}}(\Omega) \cap L^{1}_{2s}(\mathbb{R}^n)$ is a {\it weak solution} of $\mathcal{L}u=f$ in $\Omega$ if \eqref{eq-sol} holds for all $\varphi \in C^\infty_c(\Omega)$. A function $u$ is \emph{$\mathcal{L}$-harmonic} in $\Omega$ if it is a weak solution of $\mathcal{L}u=0$ in $\Omega$ and $u \in C(\Omega)$.
\end{definition}

It is useful to deal with a larger class of test functions than $C^\infty_c(\Omega)$. For this purpose, we recall the space
\begin{equation*}
V^{s, 2}(\Omega) \coloneqq \left\{ u~\text{measurable}: u|_\Omega \in L^2(\Omega) \text{ and } \int_{\Omega} \int_{\mathbb{R}^n} \frac{|u(x)-u(y)|^2}{|x-y|^{n+2s}} \,\mathrm{d}y \,\mathrm{d}x < \infty \right\},
\end{equation*}
equipped with the norm
\begin{align*}
\|u\|_{V^{s, 2}(\Omega)}
&\coloneqq \left( \|u\|_{L^2(\Omega)}^2 + [u]_{V^{s, 2}(\Omega)}^2 \right)^{1/2} \\
&\coloneqq \left( \int_\Omega |u(x)|^2 \,\mathrm{d}x + \int_\Omega \int_{\mathbb{R}^n} \frac{|u(x)-u(y)|^2}{|x-y|^{n+2s}} \,\mathrm{d}y\,\mathrm{d}x \right)^{1/2}.
\end{align*}
A slight modification of the proofs of Proposition~2.2 and Corollary~2.3 in Bj\"orn--Bj\"orn--Kim~\cite{BBK24} yields the following results.

\begin{proposition}\label{prop-test-V}
Let $u \in W^{s, 2}_{\mathrm{loc}}(\Omega) \cap L^1_{2s}(\mathbb{R}^n)$ and $f \in L^\infty(\Omega)$. Then $u$ is a weak solution \textup{(}resp.\ weak supersolution\textup{)} of $\mathcal{L}u=f$ in $\Omega$ if and only if
\begin{equation*}
\mathcal{E}(u, \varphi) \geq \int_{\Omega} f \varphi \,\mathrm{d}x
\end{equation*}
for all $\varphi \in V^{s, 2}(\Omega)$ \textup{(}resp.\ for all $0 \leq \varphi \in V^{s, 2}(\Omega)$\textup{)} with $\supp\varphi \Subset \Omega$.
\end{proposition}

\begin{proof}
By mollification, there exist functions $\varphi_j \in C^\infty_c(\Omega)$ such that $\supp\varphi_j$ is uniformly bounded in $j$ and $\varphi_j \to \varphi$ in $V^{s, 2}(\Omega)$ as $j \to \infty$. Moreover, $\varphi_j \geq 0$ if $\varphi \geq 0$. Then, Equation~(2.9) in our earlier work~\cite{KL23} shows that
\begin{equation}\label{eq-conv1}
|\mathcal{E}(u, \varphi_j - \varphi)| \leq C(u) \|\varphi_j - \varphi\|_{V^{s, 2}(\Omega)},
\end{equation}
where $C(u)$ is a constant depending on $u$, but not on $j$. Moreover, we have
\begin{equation}\label{eq-conv2}
\left| \int_{\Omega} f(\varphi_j-\varphi) \,\mathrm{d}x \right| \leq \|f\|_{L^\infty(\Omega)} |\supp(\varphi_j-\varphi)|^{1/2} \|\varphi_j-\varphi\|_{L^2(\Omega)} \quad\text{for all }j \in \mathbb{N}.
\end{equation}
Since $\sup_{j \in \mathbb{N}} |\supp(\varphi_j-\varphi)| < \infty$, the desired result follows by letting $j \to \infty$ in \eqref{eq-conv1} and \eqref{eq-conv2}.
\end{proof}

\begin{corollary}
A function $u$ is a weak solution of $\mathcal{L}u=f$ in $\Omega$ if and only if both $u$ and $-u$ are weak supersolutions of $\mathcal{L}u=f$ in $\Omega$.
\end{corollary}

When weak solutions are more regular than $W^{s, 2}_{\mathrm{loc}}(\Omega) \cap L^1_{2s}(\mathbb{R}^n)$, then the space of test functions can be further enlarged. We consider the spaces
\begin{align*}
V^{s, 2}_0(\Omega) &\coloneqq \overline{C^\infty_c(\Omega)}^{V^{s, 2}(\Omega)}, \\
W^{s, 2}_\Omega(\mathbb{R}^n) &\coloneqq \{u \in W^{s, 2}(\mathbb{R}^n): u=0\text{ a.e.\ on }\Omega^c\} = \{u \in V^{s, 2}(\mathbb{R}^n): u=0\text{ a.e.\ on }\Omega^c\}.
\end{align*}
It is clear that $V^{s, 2}_0(\Omega) \subset W^{s, 2}_\Omega(\mathbb{R}^n)$, but the equality does not hold in general; see Fiscella--Servadei--Valdinoci~\cite{FSV15}. The spaces coincide if the boundary of $\Omega$ satisfies certain regularity conditions.

\begin{lemma}\cite{FSV15}\label{lem-test}
Let $\Omega \subset \mathbb{R}^n$ be a bounded open set with a continuous boundary. Then
\begin{equation*}
V^{s, 2}_0(\Omega) = W^{s, 2}_\Omega(\mathbb{R}^n).
\end{equation*}
\end{lemma}

The space $V^{s, 2}_0(\Omega)$ can be used as a space of test functions for equation $\mathcal{L}u=f$, provided that $u$ is sufficiently regular. A slight modification of the proof of \cite[Proposition~2.4]{BBK24} yields the following result.

\begin{proposition}\label{prop-test}
Let $\Omega, \Omega' \subset \mathbb{R}^n$ be such that $\Omega \Subset \Omega'$. Let $f \in L^\infty(\Omega)$ and $u \in W^{s, 2}(\Omega') \cap L^1_{2s}(\mathbb{R}^n)$. Then $u$ is a weak solution (resp.\ weak supersolution) of $\mathcal{L}u=f$ in $\Omega$ if and only if
\begin{equation*}
\mathcal{E}(u, \varphi) \geq \int_\Omega f\varphi \,\mathrm{d}x
\end{equation*}
for all $\varphi \in V^{s, 2}_0(\Omega)$ \textup{(}resp.\ for all nonnegative $\varphi \in V^{s, 2}_0(\Omega)$\textup{)}.
\end{proposition}

In Proposition~\ref{prop-test}, $\Omega$ is bounded by assumption, but $\Omega'$ may be unbounded.

\begin{proof}
Suppose that $u$ is a weak supersolution of $\mathcal{L}u=f$ in $\Omega$ and let $0 \leq \varphi \in V_0^{s, 2}(\Omega)$. Then there exists a sequence of functions $\{\varphi_j\} \subset C_c^{\infty}(\Omega)$ such that $\varphi_j \to \varphi$ in $V^{s, 2}(\Omega)$ as $j \to \infty$. Since $(\varphi_j)_+=0=\varphi$ a.e.\ on $\Omega^c$, we have
\begin{equation*}
[(\varphi_j)_+-\varphi]_{W^{s, 2}(\Omega')}^2 \leq [(\varphi_j)_+-\varphi]_{W^{s, 2}(\mathbb{R}^n)}^2 \leq 2[(\varphi_j)_+-\varphi]_{V^{s, 2}(\Omega)}^2.
\end{equation*}
This together with Lemma~2.5 in \cite{BBK24} shows that $(\varphi_j)_+ \to \varphi$ in $W^{s, 2}(\Omega')$ as $j \to \infty$.

We now have
\begin{align*}
|\mathcal{E}(u, (\varphi_j)_+)-\mathcal{E}(u, \varphi)|
&\leq \Lambda \int_\Omega \int_{\Omega'} \frac{|u(x)-u(y)||((\varphi_j)_+-\varphi)(x) - ((\varphi_j)_+-\varphi)(y)|}{|x-y|^{n+2s}} \,\mathrm{d}y \,\mathrm{d}x \\
&\quad + \Lambda \int_\Omega \int_{\mathbb{R}^n \setminus \Omega'} \frac{(|u(x)|+|u(y)|) |((\varphi_j)_+-\varphi)(x)|}{|x-y|^{n+2s}} \,\mathrm{d}y \,\mathrm{d}x \\
&\eqqcolon I_1 + I_2.
\end{align*}
Let $d=\mathrm{dist}(\Omega, \partial\Omega')>0$ and $R=\sup_{x \in \Omega}|x|<\infty$. Then
\begin{equation*}
I_1 \leq \Lambda [u]_{W^{s, 2}(\Omega')} [(\varphi_j)_+-\varphi]_{W^{s, 2}(\Omega')}
\end{equation*}
and
\begin{align*}
I_2
&\leq \Lambda \int_\Omega |u(x)| |((\varphi_j)_+-\varphi)(x)| \int_{\mathbb{R}^n \setminus B_d(x)} \frac{\mathrm{d}y}{|x-y|^{n+2s}} \,\mathrm{d}x \\
&\quad + \Lambda \int_{\Omega} |((\varphi_j)_+-\varphi)(x)| \left( \int_{B_{2R}} \frac{|u(y)|}{d^{n+2s}} \,\mathrm{d}y + \int_{B_{2R}^c} \frac{|u(y)|}{(|y|/2)^{n+2s}} \,\mathrm{d}y \right) \mathrm{d}x \\
&\leq \Lambda \left( \frac{|\mathbb{S}^{n-1}|}{2s d^{2s}} \|u\|_{L^2(\Omega)} + \frac{|\Omega|^{1/2}}{d^{n+2s}} \|u\|_{L^1(B_{2R})} + \frac{2^{n} |\Omega|^{1/2}}{R^{2s}} \mathrm{Tail}(u; 0, 2R) \right) \|(\varphi_j)_+ - \varphi\|_{L^2(\Omega)}.
\end{align*}
Moreover,
\begin{equation*}
    \left| \int_\Omega f ((\varphi_j)_+ - \varphi) \,\mathrm{d}x \right| \leq \|f\|_{L^\infty(\Omega)} |\Omega|^{1/2} \|(\varphi_j)_+ - \varphi\|_{L^2(\Omega)}.
\end{equation*}
Since $\supp (\varphi_j)_+ \Subset \Omega$, it follows from Proposition~\ref{prop-test-V} that
\begin{equation*}
    \mathcal{E}(u, \varphi)=\lim_{j \to \infty}\mathcal{E}(u, (\varphi_j)_+) \geq \lim_{j \to \infty}\int_{\Omega}f(\varphi_j)_+ \,\mathrm{d}x =\int_{\Omega}f\varphi \,\mathrm{d}x. \qedhere
\end{equation*}
\end{proof}

The following lemma provides a useful tool for testing the equation.

\begin{lemma}\label{lem-test-V}
Let $\Omega \subset \mathbb{R}^n$ be a bounded open set with a continuous boundary and let $\Omega' \subset \mathbb{R}^n$ be an open set such that $\Omega \Subset \Omega'$. If $u \in W^{s, 2}(\Omega')$ and $u=0$ a.e.\ outside $\Omega$, then $u \in V^{s, 2}_0(\Omega)$.
\end{lemma}

\begin{proof}
Let $d$ denote the distance between $\Omega$ and $\partial \Omega'$. Then
\begin{align*}
    [u]_{V^{s, 2}(\Omega)}^2
    &\leq [u]_{W^{s, 2}(\Omega')}^2 + \int_{\Omega} \int_{\mathbb{R}^n \setminus \Omega'} \frac{|u(x)|^2}{|x-y|^{n+2s}} \,\mathrm{d}y \,\mathrm{d}x \\
    &\leq [u]_{W^{s, 2}(\Omega')}^2 + \int_{\Omega} |u(x)|^2 \int_{\mathbb{R}^n \setminus B_d(x)} \frac{\mathrm{d}y}{|x-y|^{n+2s}} \,\mathrm{d}x \\
    &= [u]_{W^{s, 2}(\Omega')}^2 + \frac{|\mathbb{S}^{n-1}|}{2s d^{2s}} \|u\|_{L^2(\Omega)}^2 < \infty.
\end{align*}
This shows that $u \in V^{s, 2}(\Omega)$. The desired result follows from Lemma~\ref{lem-test}.
\end{proof}

The following lemma will also be useful.

\begin{lemma}\cite[Lemma~7]{KKP17}\label{lem-truncation}
Let $u$ be a weak supersolution of $\mathcal{L}u=0$ in $\Omega$. Then, for any $k \in \mathbb{R}$, the function $\min\{u, k\}$ is also a weak supersolution in $\Omega$.
\end{lemma}

We also recall the definition of the $\mathcal{L}$-super- and $\mathcal{L}$-subharmonic functions from \cite[Definition~8.1]{BBK24}. See also \cite[Definition~1]{KKP17}.

\begin{definition}\label{def-superharmonic}
A measurable function $u: \mathbb{R}^n \to [-\infty, \infty]$ is \emph{$\mathcal{L}$-superharmonic} in $\Omega$ if it satisfies the following:
\begin{enumerate}[(i)]
\item $u < \infty$ a.e.\ in $\mathbb{R}^n$ and $u>-\infty$ everywhere in $\Omega$,
\item $u$ is lower semicontinuous in $\Omega$,
\item if $G \Subset \Omega$ is an open set and $v \in C(\overline{G})$ is a weak solution of $\mathcal{L}v=0$ in $G$ such that $v_+ \in L^{\infty}(\mathbb{R}^n)$ and $u \geq v$ on $G^c$, then $u \geq v$ in $G$,
\item $u_- \in L_{2s}^1(\mathbb{R}^n)$.
\end{enumerate}
A function $u$ is \emph{$\mathcal{L}$-subharmonic} in $\Omega$ if $-u$ is $\mathcal{L}$-superharmonic in $\Omega$.
\end{definition}

By \cite[Corollary~7]{KKP17}, a function $u$ is $\mathcal{L}$-harmonic in $\Omega$ if and only if it is both $\mathcal{L}$-super- and $\mathcal{L}$-subharmonic in $\Omega$. We end the section with the following result from \cite{KKP17}. The fractional Sobolev spaces $W^{\sigma, q}$ below are defined analogously to $W^{s, 2}$.

\begin{theorem}\cite[Theorem~1(ii)]{KKP17}\label{thm-integrability}
If $u$ is $\mathcal{L}$-superharmonic in $\Omega$, then $u \in W^{\sigma, q}_{\mathrm{loc}}(\Omega) \cap L^t_{\mathrm{loc}}(\Omega) \cap L^{1}_{2s}(\mathbb{R}^n)$ for any
\begin{equation*}
   \sigma \in (0, s), \quad q \in (0, n/(n-s)), \quad\text{and}\quad t \in
   \begin{cases}
       (0, \frac{n}{n-2s}) &\text{if }2s<n, \\
       (0, \infty) &\text{if }2s\geq n.
   \end{cases}
\end{equation*}
\end{theorem}

\section{Singular solutions to the fractional Laplace equation}\label{sec-fL}

In this section, we prove Theorem~\ref{thm-sing-Liouv-s}, the Liouville theorem for singular solutions to the fractional Laplace equation. As mentioned in the introduction, the main ingredients of the proof are a nonlocal Caccioppoli--Cimmino--Weyl lemma, the classical Liouville theorem, and a B\^ocher type Theorem~\ref{thm-weak-Bocher}. The proof of Theorem~\ref{thm-weak-Bocher} will be provided in Section~\ref{sec-weak-Bocher}. We recall the Caccioppoli--Cimmino--Weyl lemma and the Liouville theorem for the fractional Laplacian in the following theorems.

\begin{theorem}[Nonlocal Caccioppoli--Cimmino--Weyl lemma]\cite[Theorem~12.17]{Gar19}\label{thm-Weyl}
Let $\Omega \subset \mathbb{R}^n$ be a nonempty open set. If $u \in L^1_{2s}(\mathbb{R}^n)$ satisfies $(-\Delta)^su=0$ in $\mathscr{D}'(\Omega)$, then $u$ has a representative that belongs to $C^\infty(\Omega)$.
\end{theorem}

\begin{theorem}[Liouville theorem for the fractional Laplacian]\cite[Lemma~3.2]{BKN02}\label{thm-Liouville-s}
If a nonnegative function $u \in L^1_{2s}(\mathbb{R}^n)$ satisfies $(-\Delta)^su=0$ in $\mathscr{D}'(\mathbb{R}^n)$, then $u$ must be constant.
\end{theorem}

See also \cite{AdTEJ20,CDL15,Fal16,FW16,GK25} for more general Liouville theorems. We now turn to the proof of Theorem~\ref{thm-sing-Liouv-s}.

\begin{proof}[Proof of Theorem~\ref{thm-sing-Liouv-s}]
We may assume that $u$ is bounded from below in a neighborhood of the origin, say in $B_r=B_r(0)$ for some $r>0$. By Theorem~\ref{thm-weak-Bocher}, there exists a constant $\bar{a} \geq 0$ such that $(-\Delta)^su=\bar{a}\delta_0$ in $B_r$ in the sense that
\begin{equation}\label{eq-fL-weak}
    \int_{\mathbb{R}^n} \int_{\mathbb{R}^n} \frac{(u(x)-u(y))(\varphi(x)-\varphi(y))}{|x-y|^{n+2s}} \,\mathrm{d}y \,\mathrm{d}x = \bar{a} \varphi(0)
\end{equation}
for any $\varphi \in C^\infty_c(B_r)$. Since $u$ is $(-\Delta)^s$-harmonic in $\mathbb{R}^n \setminus \{0\}$, \eqref{eq-fL-weak} holds for any $\varphi \in C^\infty_c(\mathbb{R}^n)$. In particular, this implies that $(-\Delta)^su=\bar{a}\delta_0$ in $\mathscr{D}'(\mathbb{R}^n)$ since the left-hand side of \eqref{eq-fL-weak} is equal to
\begin{equation*}
    \int_{\mathbb{R}^n} u(x)(-\Delta)^s \varphi(x) \,\mathrm{d}x.
\end{equation*}
Notice that this quantity is finite for any $\varphi \in C^\infty_c(\mathbb{R}^n)$ since
\begin{equation*}
    \left| \int_{\mathbb{R}^n} u(x)(-\Delta)^s \varphi(x) \,\mathrm{d}x \right| \leq C \int_{\mathbb{R}^n} \frac{|u(x)|}{(1+|x|)^{n+2s}} \,\mathrm{d}x < \infty
\end{equation*}
for some $C=C(n, s, \varphi)>0$, where we used \cite[Proposition~2.9]{Gar19} and the fact that $(-\Delta)^s\varphi \in C^\infty(\mathbb{R}^n)$.

Let $\Phi_s$ be given by \eqref{eq-FS}, which is the fundamental solution for $(-\Delta)^s$, i.e.\ $(-\Delta)^s\Phi_s=\delta_0$ in $\mathscr{D}'(\mathbb{R}^n)$. Then
\begin{equation*}
    (-\Delta)^s(u-\bar{a}\Phi_s)=0\quad\text{in }\mathscr{D}'(\mathbb{R}^n).
\end{equation*}
Thus, Theorem~\ref{thm-Weyl} shows that $v\coloneqq u-\bar{a}\Phi_s \in C^\infty(\mathbb{R}^n)$, after redefinition on a set of measure zero if necessary. In particular, $v$ is $(-\Delta)^s$-harmonic in $\mathbb{R}^n$.

Since $u$ is bounded from one side in a neighborhood of infinity, say in $B_R^c$ for some $R>0$, and $\Phi_s$ is bounded in $B_R^c$, $v$ is also bounded from one side in $B_R^c$. Moreover, since $v \in C^\infty(\mathbb{R}^n)$, $v$ is bounded from one side in $\mathbb{R}^n$. Thus, it follows from Theorem~\ref{thm-Liouville-s} that $v$ is constant. Therefore,
\begin{equation*}
    u(x)=a|x|^{2s-n}+b,
\end{equation*}
where $a=\frac{\Gamma(\frac{n}{2}-s)}{4^{s}\pi^{n/2}\Gamma(s)}\bar{a}$ and $b=v$ are constants.
\end{proof}

\begin{remark}
    If $u$ is assumed to be bounded from above or below in $\mathbb{R}^n$ in Theorem~\ref{thm-sing-Liouv-s}, then B\^ocher type results for \emph{globally nonnegative} functions in the literature (see, for instance, Li--Wu--Xu~\cite[Theorem~4]{LWX18}, Li--Liu--Wu--Xu~\cite[Theorem~1.3]{LLWX20}, and Klimsiak~\cite{Kli25}) can be used instead of Theorem~\ref{thm-weak-Bocher}.
\end{remark}

For the fractional Laplacian, a classification of $(-\Delta)^s$-harmonic functions in exterior domains can be deduced from the B\^ocher Theorem~\ref{thm-Bocher} by applying the Kelvin transform. Theorem~\ref{thm-Bocher} for general operators $\mathcal{L}$ will be proved in Section~\ref{sec-Bocher}.

\begin{theorem}
Let $u$ be $(-\Delta)^s$-harmonic in a neighborhood of infinity. Then either
\begin{equation}\label{eq-limsupinf}
    \limsup_{x \to \infty} u(x) = \infty \quad\text{and}\quad \liminf_{x \to \infty} u(x) = -\infty,
\end{equation}
or
\begin{equation*}
    \lim_{x \to \infty} u(x) = c
\end{equation*}
for some constant $c \in \mathbb{R}$.
\end{theorem}

\begin{proof}
We define the Kelvin transform $\bar{u}$ of $u$ by
\begin{equation*}
    \bar{u}(\bar{x}) = \frac{u(x)}{|\bar{x}|^{n-2s}} \quad \text{for } \bar{x}=\frac{x}{|x|^2}.
\end{equation*}
Then $\bar{u}$ is $(-\Delta)^s$-harmonic in $\Omega \setminus \{0\}$ for some neighborhood $\Omega$ of the origin. Suppose that \eqref{eq-limsupinf} does not hold. Then $u$ is bounded from above or below near infinity. We may assume without loss of generality that $u$ is bounded from below in $B_R^c$ for some ball $B_R=B_R(0)$ with a radius $R>0$ sufficiently large so that $B_{1/R} \subset \Omega$. Let $a \in \mathbb{R}$ be such that $u+a$ is positive in $B_R^c$. Then the function
\begin{equation*}
    \bar{x} \mapsto \bar{u}(\bar{x}) + \frac{a}{|\bar{x}|^{n-2s}} = \frac{u(x)+a}{|\bar{x}|^{n-2s}}
\end{equation*}
is $(-\Delta)^s$-harmonic in $B_{1/R} \setminus \{0\}$ and positive in $B_{1/R}$. Thus, Theorem~\ref{thm-Bocher} shows that
\begin{equation*}
    \frac{u(x)+a}{|\bar{x}|^{n-2s}} = \frac{b}{|\bar{x}|^{n-2s}} + v(\bar{x})
\end{equation*}
for some constant $b$ and some function $v$ that is $(-\Delta)^s$-harmonic in $B_{1/R}$. We thus get
\begin{equation*}
    u(x) = b-a + |\bar{x}|^{n-2s} v(\bar{x}),
\end{equation*}
which yields
\begin{equation*}
    \lim_{x \to \infty} u(x) = \lim_{\bar{x} \to 0} (b-a + |\bar{x}|^{n-2s} v(\bar{x})) = b-a. \qedhere
\end{equation*}
\end{proof}

\section{Weak B\^ocher type theorem}\label{sec-weak-Bocher}

In this section, we prove Theorem~\ref{thm-weak-Bocher}, which we refer to as a weak B\^ocher type theorem, since it can be viewed as a weaker version of Theorem~\ref{thm-Bocher}. Despite its weaker form, it plays a crucial role in the construction of the fundamental solution in Section~\ref{sec-fundamentalsol}.

The main difficulty in proving Theorem~\ref{thm-weak-Bocher} lies in the fact that the regularity assumption $W^{s, 2}_{\mathrm{loc}}(\Omega \setminus \{0\}) \cap L^{1}_{2s}(\mathbb{R}^n)$ for $\mathcal{L}$-harmonic functions is not sufficient for the quantity $\mathcal{E}(u, \varphi)$ to be well defined. Nevertheless, we can show that singular harmonic functions in $\Omega \setminus \{0\}$ are $\mathcal{L}$-superharmonic in $\Omega$, and thus enjoy some local regularity in $\Omega$ from Theorem~\ref{thm-integrability}. Establishing this is our first step.

\begin{lemma}\label{lem-superharmonic}
Let $\Omega \subset \mathbb{R}^n$ be an open set containing the origin. Let $u$ be $\mathcal{L}$-harmonic in $\Omega \setminus \{0\}$ and bounded from below in $\Omega$. Then, after a redefinition at the origin if necessary, $u$ is $\mathcal{L}$-superharmonic in $\Omega$.
\end{lemma}

To prove Lemma~\ref{lem-superharmonic}, we need the following auxiliary result.

\begin{lemma}\label{lem-supersolution}
Let $\Omega \subset \mathbb{R}^n$ be an open set containing the origin. Let $u$ be a weak supersolution of 
\begin{equation}\label{eq-main-sec4}
\mathcal{L}u=0
\end{equation}
in $\Omega \setminus \{0\}$ and bounded from below in $\Omega$. Then $\min\{u, j\}$, $j \in \mathbb{N}$, is a weak supersolution of \eqref{eq-main-sec4} in $\Omega$.
\end{lemma}

\begin{proof}
It follows from Lemma~\ref{lem-truncation} that $u_j \coloneqq \min\{u, j\}$ is a weak supersolution of \eqref{eq-main-sec4} in $\Omega \setminus \{0\}$. Thus, it is enough to show that $u_j$ is a weak supersolution of \eqref{eq-main-sec4} in $\Omega$. This can be done by following the proof of Theorem~1.1 in our earlier paper~\cite{KL24}. For completeness, we briefly outline the main steps of the proof.

Choose $R>0$ so that $B_{R}=B_{R}(0) \Subset \Omega$ and define $v_j \coloneqq \min\{u-j, 0\}=u_j-j$. We first observe that $[(v_j)_+]_{W^{s, 2}(B_R)}=0$. On the other hand, \cite[Equation~(4.6)]{KL24}, obtained by using the Caccioppoli estimate for weak supersolution $v_j$ instead of \cite[Lemma~3.3]{KL24}, shows that
\begin{equation*}
    \int_{B_{R/2}} \int_{B_{R/2}} \frac{|\bar{v}_j(x)-\bar{v}_j(y)|^2}{|x-y|^{n+2s}} \min\{\bar{v}_j^{\delta-1}(x), \bar{v}_j^{\delta-1}(y)\} \,\mathrm{d}y \,\mathrm{d}x<\infty,
\end{equation*}
where $\bar{v}_j=(v_j)_-+\mathrm{Tail}((v_j)_-; 0, R/2)$ and $\delta>0$. Here, $\delta$ can be assumed to be small so that $\delta\leq 1$ since $v_j \in L^\infty(B_R)$. This implies that $[(v_j)_-]_{W^{s, 2}(B_{R/2})} = [\bar{v}_j]_{W^{s, 2}(B_{R/2})}< \infty$, and hence $v_j \in W^{s, 2}(B_{R/2})$. In particular, it follows that $u_j \in W^{s, 2}(B_{R/2})$ as well. We now proceed as in the second part of the proof of \cite[Theorem~1.1]{KL24} to conclude that $u_j$ is a weak supersolution of \eqref{eq-main-sec4} in $\Omega$.
\end{proof}

\begin{proof}[Proof of Lemma~\ref{lem-superharmonic}]
Lemma~\ref{lem-supersolution} shows that $u_j \coloneqq \min\{u, j\}$ is a weak supersolution of \eqref{eq-main-sec4} in $\Omega$ for any $j \in \mathbb{N}$. Let $\hat{u}_j$ be the lsc-regularization of $u_j$, i.e.,
\begin{equation*}
\hat{u}_j(x) \coloneqq
\begin{cases}
    \essliminf_{y \to x} u_j(y) &\text{if }x \in \Omega, \\
    u_j(x) &\text{if }x \in \Omega^c.
\end{cases}
\end{equation*}
Then $\hat{u}_j=u_j$ a.e.\ in $\Omega$ and it is also a weak supersolution of \eqref{eq-main-sec4} in $\Omega$ by \cite[Theorem~9]{KKP17}. Moreover, Theorem~12 and Lemma~12 in \cite{KKP17} show that $u=\lim_{j\to\infty}\hat{u}_j$ (after redefinition at the origin if necessary) is $\mathcal{L}$-superharmonic in $\Omega$.
\end{proof}

Next, we recall the isolated singularity theorem developed in our earlier work \cite{KL24} as follows. We note that this theorem was indeed proved for more general class of nonlocal nonlinear equations in \cite{KL24}.

\begin{theorem}[Isolated singularity theorem]\label{thm-iso-sing}\cite[Theorem~1.3]{KL24}
Let $\Omega \subset \mathbb{R}^n$ be an open set containing the origin. Let $u$ be $\mathcal{L}$-harmonic in $\Omega \setminus \{0\}$ and bounded from below in $\Omega$. Then either $u$ has removable singularity at the origin or
\begin{equation*}
u(x) \eqsim |x|^{2s-n}
\end{equation*}
in a neighborhood of the origin.
\end{theorem}

The following lemma will also be useful in the proof of Theorem~\ref{thm-weak-Bocher}.

\begin{lemma}\cite[Lemma~5.3]{KL24}\label{lem-const}
Let $\varphi \in W^{s, 2}(\Omega) \cap L^\infty(\Omega)$ be such that $\mathrm{supp}\,{\varphi} \Subset \Omega$ and $\varphi= 1$ in a neighborhood of the origin. If $u$ is a weak solution of \eqref{eq-main-sec4} in $\Omega \setminus \{0\}$, then the value
\begin{equation*}
a \coloneqq \mathcal{E}(u, \varphi)
\end{equation*}
is finite and independent of the particular choice of $\varphi$.
\end{lemma}

We are now in a position to prove Theorem~\ref{thm-weak-Bocher}.

\begin{proof}[Proof of Theorem~\ref{thm-weak-Bocher}]
In view of Theorem~\ref{thm-iso-sing}, we can fix $R \in (0,1/2)$ so that $B_R = B_R(0) \Subset \Omega$ and that
	\begin{equation}\label{eq-asymp}
		u \eqsim |x|^{2s-n}
		\quad \text{in }B_R.
	\end{equation}
	Since being a solution is a local property and we already know that $\mathcal{L}u=0$ in $\Omega \setminus \{0\}$, it suffices to prove \eqref{eq-delta-weak} for any $\varphi \in C_c^{\infty}(B_R)$.
	
	Let $\varphi \in C_c^{\infty}(B_R)$ and let $r<R/2$. Choose a cut-off function $\eta_r \in C_c^{\infty}(B_r)$ such that $0\leq \eta_r \leq 1$ in $\mathbb{R}^n$, $\eta_r=1$ in $B_{r/2}$, and $|D\eta_r| \leq 4/r$. We define
	\begin{equation*}
		\varphi_r \coloneqq (1-\eta_r)\varphi+\eta_r \varphi(0).
	\end{equation*}
	Since $\varphi_r=\varphi(0)$ in a neighborhood of the origin, it follows from Lemma~\ref{lem-const} that
	\begin{equation*}
		\mathcal{E}(u, \varphi_r)=a \varphi(0) \quad \text{for some $a \geq 0$}.
	\end{equation*}
     Therefore, it is enough to show that $\mathcal{E}(u, \varphi_r) \to \mathcal{E}(u, \varphi)$ as $r \to 0$.
    
	Note that $\eta_r(x) \to 0$ for any $x \in B_R \setminus \{0\}$ as $r \to 0$, and that
	\begin{align*}
		\mathcal{E}(u, \varphi-\varphi_r)
        &=\mathcal{E}(u, (\varphi-\varphi(0))\eta_r)\\
		&=\int_{B_R} \int_{B_R} (u(x)-u(y))(\varphi(x)-\varphi(y))\eta_r(x) k(x, y) \,\mathrm{d}y\,\mathrm{d}x\\
		&\quad+\int_{B_R} \int_{B_R} (u(x)-u(y))(\varphi(y)-\varphi(0))(\eta_r(x)-\eta_r(y))k(x, y) \,\mathrm{d}y\,\mathrm{d}x\\
		&\quad+2\int_{B_R} \int_{B_R^c} (u(x)-u(y))(\varphi(x)-\varphi(0))\eta_r(x)k(x, y) \,\mathrm{d}y\,\mathrm{d}x \\
		&\eqqcolon I_1+I_2+I_3.
	\end{align*}
We choose $\sigma \in (0,s)$ close to $s$ so that $\sigma \geq 2s-1$. Then the integrand in $I_1$ is bounded by the function
\begin{equation*}
    \Lambda \|D\varphi\|_{L^\infty(B_R)} \frac{|u(x)-u(y)|}{|x-y|^{n+2s-1}},
\end{equation*}
which belongs to $L^1(B_R \times B_R)$ since
\begin{align*}
    \int_{B_R} \int_{B_R} \frac{|u(x)-u(y)|}{|x-y|^{n+2s-1}} \,\mathrm{d}y \,\mathrm{d}x \leq \int_{B_R} \int_{B_R} \frac{|u(x)-u(y)|}{|x-y|^{n+\sigma}} \,\mathrm{d}y \,\mathrm{d}x < \infty
\end{align*}
by Lemma~\ref{lem-superharmonic} and Theorem~\ref{thm-integrability}. Thus, $I_1 \to 0$ as $r \to 0$ by the dominated convergence theorem.

Next, we observe that for any $\theta \geq 2$ and $r \in (0, R/\theta)$,
\begin{equation*}
	\begin{aligned}
		|I_2| &\leq \Lambda \int_{B_{\theta r}} \int_{B_{\theta r}} \frac{|u(x)-u(y)||\varphi(y)-\varphi(0)||\eta_r(x)-\eta_r(y)|}{|x-y|^{n+2s}}\,\mathrm{d}y\,\mathrm{d}x\\
		&\quad+4\Lambda\|\varphi\|_{L^\infty(B_R)} \int_{B_{r}} \int_{B_R \setminus B_{\theta r}} \frac{|u(x)-u(y)|}{|x-y|^{n+2s}}\,\mathrm{d}y\,\mathrm{d}x\\
		&\eqqcolon I_{2, 1}+I_{2,2}.
	\end{aligned}
\end{equation*}
Since $|\varphi(y)-\varphi(0)| \leq \|D\varphi\|_{L^\infty(B_{\theta r})}|y| < \theta\|D\varphi\|_{L^\infty(B_{R})}r$ for $y \in B_{\theta r}$ and $|D\eta_r|\leq 4/r$, we have
\begin{equation*}
	I_{2,1} \leq 4\theta\Lambda\|D\varphi\|_{L^\infty(B_{R})} \int_{B_{\theta r}} \int_{B_{\theta r}} \frac{|u(x)-u(y)|}{|x-y|^{n+2s-1}}\,\mathrm{d}y\,\mathrm{d}x \leq 4\theta\Lambda\|D\varphi\|_{L^\infty(B_{R})}[u]_{W^{\sigma, 1}(B_{\theta r})}.
\end{equation*}
Moreover, for $x \in B_{r}$ and $y \in B_R \setminus B_{\theta r}$, we get from \eqref{eq-asymp} that
\begin{equation*}
	0<u(y) \lesssim |y|^{2s-n} \leq (\theta|x|)^{2s-n} < |x|^{2s-n},
\end{equation*}
where the comparable constants are independent of $\theta$ and $r$. Thus, we have
\begin{equation*}
	I_{2, 2} \lesssim \int_{B_{r}} \int_{B_R \setminus B_{\theta r}} \frac{u(x)+u(y)}{|x-y|^{n+2s}}\,\mathrm{d}y\,\mathrm{d}x \lesssim \int_{B_{r}} \int_{B_R \setminus B_{\theta r}} \frac{|x|^{2s-n}}{|y|^{n+2s}}\,\mathrm{d}y\,\mathrm{d}x \lesssim \theta^{-2s}.
\end{equation*}
We now combine two estimates above to derive
\begin{equation*}
    |I_2| \lesssim \theta[u]_{W^{\sigma, 1}(B_{\theta r})}+\theta^{-2s},
\end{equation*}
which implies
\begin{equation*}
    \limsup_{r \to 0}|I_2| \lesssim \theta^{-2s}.
\end{equation*}
Since $\theta\geq2$ can be arbitrarily large, we conclude that $\lim_{r \to 0}I_2=0$.

Finally, we note that the integrand in $I_3$ is bounded by the function
\begin{equation*}
    2\Lambda \|\varphi\|_{L^\infty(B_R)} \frac{|u(x)-u(y)|}{|x-y|^{n+2s}} \chi_{B_{R/2} \times B_R^c}(x, y)
\end{equation*}
since $\mathrm{supp}\,\eta_r \subset B_r \subset B_{R/2}$. This function belongs to $L^1(\mathbb{R}^n \times \mathbb{R}^n)$ since
\begin{align*}
    \int_{B_{R/2}} \int_{B_R^c} \frac{|u(x)-u(y)|}{|x-y|^{n+2s}} \,\mathrm{d}y \,\mathrm{d}x
    &\leq \int_{B_{R/2}} \int_{B_R^c} \frac{|u(x)|+|u(y)|}{(|y|/2)^{n+2s}} \,\mathrm{d}y \,\mathrm{d}x \\
    &\lesssim R^{n-2s} \left( \fint_{B_{R/2}} |u| \,\mathrm{d}x + \mathrm{Tail}(u, 0, R) \right) < \infty.
\end{align*}
Thus, $I_3 \to 0$ as $r \to 0$ by the dominated convergence theorem, and the proof is complete.
\end{proof}

We remark that the proof of Theorem~\ref{thm-weak-Bocher} can easily be extended to more general class of nonlocal nonlinear operators studied in \cite{KL24}.

\section{Fundamental solution}\label{sec-fundamentalsol}

This section is devoted to the proofs of existence, uniqueness, and some useful properties of the fundamental solution for $\mathcal{L}$. One natural approach is to first prove the existence and asymptotic properties of the heat kernel for the associated parabolic operator $\partial_t + \mathcal{L}$ and then integrate it in time. The existence and pointwise estimates for the heat kernel are now well established; see, for example, Bass--Levin~\cite{BL02} and Chen--Kumagai~\cite{CK03}. See also \cite{BGK09,CKW21,GHH17,GHH18,GHL14,LW25} for related results in more general settings. Alternatively, one can directly construct the fundamental solution as the limit of the Green functions for $\mathcal{L}$ in balls as the radius tends to infinity. In this work, we adopt the latter approach, as it provides detailed information on regularity and yields a notion of solution that is suitable for our purpose.

When $\mathcal{L}$ is given by the fractional Laplacian $(-\Delta)^s$, the fundamental solution is understood as the distributional solution of $(-\Delta)^s \Phi_s = \delta_0$ in $\mathscr{D}'(\mathbb{R}^n)$ (see Definition~\ref{def-distribution}), but this definition is not appropriate for general operators $\mathcal{L}$ with measurable kernels $k$ because $\mathcal{L}\varphi$ is not well defined in general even for nice test functions $\varphi \in C^\infty_c(\mathbb{R}^n)$. Instead, we define the fundamental solution for $\mathcal{L}$ by the function $\Phi$ satisfying the properties \eqref{eq-FS1}--\eqref{eq-FS3} in Theorem~\ref{thm-FS}. Note that this definition is consistent with the standard distributional definition in the case of the fractional Laplacian, as we have already observed in the proof of Theorem~\ref{thm-sing-Liouv-s}.

We begin by recalling the definition of the Green function of $\mathcal{L}$ on balls, following Kassmann--Kim--Lee~\cite{KKL23}, who adopt the classical definition of Littman--Stampacchia--Weinberger~\cite{LSW63}.

\begin{definition}\cite[Definition~1.2]{KKL23}
A measurable function $G^R: \mathbb{R}^n \times \mathbb{R}^n \to [0, \infty]$ is called a \emph{Green function} of $\mathcal{L}$ on $B_R$ if $G^R(\cdot, y) \in L^1(B_R)$ for each $y \in B_R$, $G^R=0$ a.e.\ on $(\mathbb{R}^n \times \mathbb{R}^n) \setminus (B_R \times B_R)$, and
\begin{equation*}
    \int_{\mathbb{R}^n} G^R(x, y)\psi(x) \,\mathrm{d}x = \varphi(y)
\end{equation*}
for every $y \in B_R$, $\varphi \in V^{s, 2}_0(B_R) \cap C(B_R)$\footnote{In \cite{KKL23}, the function space $H_{B_R}(\mathbb{R}^n; k) \coloneqq \{u \in W^{s, 2}(\mathbb{R}^n): u=0\text{ a.e.\ on }B_R^c\}$ is used instead of $V^{s, 2}_0(B_R)$. However, Lemma~\ref{lem-test} shows that this space coincides with the space $V^{s, 2}_0(B_R)$ since $B_R$ has a smooth boundary.}, and $\psi \in L^\infty(B_R)$ with
\begin{equation*}
    \mathcal{E}(\varphi, v) = \langle \psi, v \rangle \quad\text{for all } v \in V^{s, 2}_0(B_R).
\end{equation*}
\end{definition}

The main results established in \cite{KKL23} include the following\footnote{In \cite{KKL23}, the authors established the Green function estimates \eqref{eq-Green-est}, which recover the classical result $G(x, y) \eqsim |x-y|^{2-n}$ in the limit $s \nearrow 1$, by carefully investigating the comparison constants in \eqref{eq-Green-est}. To ensure this robustness, they assumed $n>2$. However, the (non-robust) estimates \eqref{eq-Green-est} can still be derived from the same arguments under the weaker condition $n > 2s$.}.

\begin{theorem}\label{thm-Green}
There exists a unique Green function $G^R$ of $\mathcal{L}$ on $B_R$. Moreover, $G^R$ satisfies
\begin{align}\label{eq-Green-est}
\begin{split}
G^R(x, y) &\leq C |x-y|^{2s-n} \quad\text{for all }x, y \in \mathbb{R}^n, \\
G^R(x, y) &\geq c |x-y|^{2s-n} \quad\text{for all }x, y \in B_R \text{ with }|x-y| \leq \mathrm{dist}(y, \partial B_R)/2,
\end{split}
\end{align}
where constants $C\geq c>0$ depend only on $n$, $s$, and $\Lambda$.
\end{theorem}

In the same spirit as Theorem~\ref{thm-weak-Bocher}, the Green function $G^R$ satisfies the following property.

\begin{lemma}\label{lem-Green-weak}
Let $G^R$ be the Green function of $\mathcal{L}$ on $B_R$, and let $y \in B_R$. Then $G^R$ satisfies
\begin{equation*}
\mathcal{E}(G^R(\cdot, y), \varphi) = \varphi(y)
\end{equation*}
for any $\varphi \in C^\infty_c(B_R)$. Moreover, $G^R(\cdot, y) \in W^{s, 2}_{\mathrm{loc}}(B_R \setminus \{y\}) \cap L^1_{2s}(\mathbb{R}^n)$. In particular, $G^R(\cdot, y)$ is a weak solution of $\mathcal{L}G^R(\cdot, y)=0$ in $B_R \setminus \{y\}$.
\end{lemma}

\begin{proof}
We recall from the proof of Theorem~1.3 in \cite{KKL23} that the Green function $G^R$ is obtained as the limit of $G_\rho^R$ as $\rho \to 0$, where $G_\rho^R \in V^{s, 2}_0(B_R)$ is the regularized Green function satisfying
\begin{equation}\label{eq-reg-Green}
\mathcal{E}(G_\rho^R(\cdot, y), \varphi)=\fint_{B_\rho(y)} \varphi \,\mathrm{d}x \quad\text{for all }\varphi \in C^\infty_c(B_R).
\end{equation}
Note that the sequence $\{G_\rho^R(\cdot, y)\}_\rho$ is uniformly bounded in $W^{\sigma, q}(\mathbb{R}^n)$ for any $\sigma \in (0,s)$ and $q \in [1, n/(n-s))$; see \cite[p.~16973]{KKL23}.

Now, we take $\sigma \in (0, s)$ sufficiently close to $s$ so that $2s-\sigma<1$, and fix any exponent $q \in (1, n/(n-s))$. Then there exists a sequence $\rho_j \to 0$ such that $G_{\rho_j}^R(\cdot, y)$ converges weakly to $G^R(\cdot, y)$ in $W^{\sigma, q}(\mathbb{R}^n)$ as $j \to \infty$. Note that, for any $\varphi \in C^\infty_c(B_R)$, the map $u \mapsto \mathcal{E}(u, \varphi)$ is a continuous functional on $W^{\sigma, q}(\mathbb{R}^n)$ since
\begin{align*}
|\mathcal{E}(u, \varphi)|
&\leq \Lambda \int_{\mathbb{R}^n} \int_{\mathbb{R}^n} \frac{|u(x)-u(y)||\varphi(x)-\varphi(y)|}{|x-y|^{n+2s}} \,\mathrm{d}y \,\mathrm{d}x \\
&\leq \Lambda [u]_{W^{\sigma, q}(\mathbb{R}^n)} [\varphi]_{W^{2s-\sigma, q/(q-1)}(\mathbb{R}^n)}.
\end{align*}
Therefore, the left-hand side of \eqref{eq-reg-Green} (with $\rho$ replaced by $\rho_j$) converges to $\mathcal{E}(G^R(\cdot, y), \varphi)$ as $j \to \infty$. It is clear that the right-hand side of \eqref{eq-reg-Green} (with $\rho$ replaced by $\rho_j$) converges to $\varphi(0)$ as $j \to \infty$.

The second assertion follows from Lemma~9.2 in Kim--Weidner~\cite{KW24a}.
\end{proof}

\begin{proof}[Proof of Theorem~\ref{thm-FS}]
Let $G^R$ be the Green function for $\mathcal{L}$ in $B_R=B_R(0)$ with the pointwise estimates \eqref{eq-Green-est}. Then, Lemma~\ref{lem-Green-weak} shows that $G^R(\cdot, 0)$ is a weak solution of $\mathcal{L}G^R(\cdot, 0)=0$ in $B_R \setminus \{0\}$. Thus, it follows from the uniform H\"older estimate (see, for instance, Di Castro--Kuusi--Palatucci~\cite[Theorem~1.2]{DCKP16} or Cozzi~\cite[Theorem~6.4]{Coz17}) and the pointwise upper bound in \eqref{eq-Green-est} that
\begin{align*}
[G^R(\cdot, 0)]_{C^\alpha(\overline{B}_r(x_0))}
&\leq \frac{C}{r^\alpha} \left( \|G^R(\cdot, 0)\|_{L^\infty(B_{2r}(x_0))} + \mathrm{Tail}(G^R(\cdot, 0); x_0, r) \right) \\
&\leq \frac{C}{r^\alpha} \left( r^{2s-n} + r^{2s} \int_{\mathbb{R}^n \setminus B_r(x_0)} \frac{r^{2s-n}}{|x-x_0|^{n+2s}} \,\mathrm{d}x \right) \\
&\leq C r^{2s-n-\alpha},
\end{align*}
provided that $B_{4r}(x_0) \subset B_R \setminus \{0\}$, where $C=C(n, s, \Lambda)>0$. Applying the Arzel\`a--Ascoli theorem together with a standard diagonal argument shows that there exist a sequence $R_j \to \infty$ and a function $\Phi \in C(\mathbb{R}^n \setminus \{0\})$ such that $G^{R_j}(\cdot, 0) \to \Phi$ uniformly on every compact subset of $\mathbb{R}^n \setminus \{0\}$ as $j \to \infty$. Moreover, it follows from \eqref{eq-Green-est} that
\begin{align}
\Phi(x) &\leq C |x|^{2s-n} \quad\text{for all }x \in \mathbb{R}^n, \nonumber\\
\Phi(x) &\geq c |x|^{2s-n} \quad\text{for all }x\text{ in a neighborhood of the origin}, \label{eq-Gamma-lower}
\end{align}
with the same constants $C\geq c>0$, which gives \eqref{eq-FS-bounds}. As a consequence, we also obtain \eqref{eq-FS3}. Note that $|x|^{2s-n} \in L^1_{2s}(\mathbb{R}^n)$.

Now, it follows from the convergence theorem in Korvenp\"a\"a--Kuusi--Palatucci~\cite[Corollary~5]{KKP17} that $\Phi$ is $\mathcal{L}$-harmonic in $\mathbb{R}^n \setminus \{0\}$. By Lemma~\ref{lem-superharmonic}, it is also $\mathcal{L}$-superharmonic in $\mathbb{R}^n$. Thus, \eqref{eq-FS2} holds. Finally, Theorem~\ref{thm-weak-Bocher} shows that there exists $a \geq 0$ such that
\begin{equation*}
\int_{\mathbb{R}^n} \int_{\mathbb{R}^n} (\Phi(x)-\Phi(y))(\varphi(x)-\varphi(y)) k(x, y) \,\mathrm{d}y \,\mathrm{d}x= a\varphi(0)
\end{equation*}
for any $\varphi \in C^\infty_c(\mathbb{R}^n)$. Since $\Phi$ has a singularity at the origin by \eqref{eq-Gamma-lower}, it follows that $a>0$. We may assume that $a=1$ by considering $\Phi/a$ instead of $\Phi$, if necessary. This shows \eqref{eq-FS1}. The uniqueness will follow from Theorem~\ref{thm-sing-Liouv-L}.
\end{proof}

\section{B\^ocher type theorem}\label{sec-Bocher}

In this section, we prove the B\^ocher type Theorem~\ref{thm-Bocher} by refining the isolated singularity Theorem~\ref{thm-iso-sing}. A key ingredient in the proof is the localized maximum principle stated in Theorem~\ref{thm-loc-max} (or the localized comparison principle in Corollary~\ref{cor-loc-comp}). Before proving Theorem~\ref{thm-loc-max}, we first display the standard maximum principle in the nonlocal context. See, for instance, Silvestre~\cite[Proposition~2.21]{Sil07}, Lindgren--Lindqvist~\cite[Lemma~9]{LL14}, Korvenp\"a\"a--Kuusi--Palatucci~\cite[Lemma~6]{KKP17}, Fern\'andez-Real--Ros-Oton~\cite[Lemma~2.3.3 and Corollary 2.3.4]{FRRO24}, Bj\"orn--Bj\"orn--Kim~\cite[Corollary~3.6]{BBK24}, or Kim--Lee~\cite[Lemma~5.2]{KL23} for its proof.

\begin{theorem}[Maximum principle]\label{thm-MP}
Let $\Omega, \Omega' \subset \mathbb{R}^n$ be nonempty open sets such that $\Omega \Subset \Omega'$. Let $u \in W^{s, 2}(\Omega')$ be a weak supersolution of $\mathcal{L}u=0$ in $\Omega$ such that $u_- \in V^{s, 2}_0(\Omega)$. Then $u \geq 0$ a.e.\ in $\Omega$.
\end{theorem}

Recall that, in the local case, a standard proof of B\^ocher type theorem relies on the (extended form of) comparison principle between a singular harmonic function and the fundamental solution on shrinking annuli; see, for instance, the proof of Theorem~5 in Gilbarg--Serrin~\cite{GS56}. This method requires comparing the two functions on the boundary of each annulus.

In the nonlocal setting, however, the maximum principle in Theorem~\ref{thm-MP} requires pointwise control of the functions in the complement of the domain (note that $u_- \in V^{s, 2}_0(\Omega)$ implies $u \geq 0$ a.e.\ in $\Omega^c$). This condition is too restrictive for our purpose, since we wish to compare a singular solution with the fundamental solution without any a priori information about which one dominates the other outside the domain. For this reason, we localize Theorem~\ref{thm-MP} to obtain Theorem~\ref{thm-loc-max}.

We first establish the following lemma, which captures the effect of a non-homogeneous term in the maximum principle for nonlocal operators.

\begin{lemma}\label{lem-MP-nonhom}
Let $0<r<R$ and let $\Omega' \subset \mathbb{R}^n$ be an open set such that $B_R \setminus \overline{B}_r \Subset \Omega'$. Let $f \in L^\infty(B_R \setminus \overline{B}_r)$ and let $u \in W^{s, 2}(\Omega')$ be a weak supersolution of $\mathcal{L}u=f$ in $B_R \setminus \overline{B}_r$ such that $u_- \in V^{s, 2}_0(B_R \setminus \overline{B}_r)$. Then there exists a constant $C=C(n, s, \Lambda)>0$ such that
\begin{equation*}
    u \geq -CR^{2s} \|f\|_{L^\infty(B_R \setminus \overline{B}_r)}
\end{equation*}
a.e.\ in $B_R \setminus \overline{B}_r$.
\end{lemma}

\begin{proof}
Let $C_0 = \|f\|_{L^\infty(B_R \setminus \overline{B}_r)}$. It is enough to prove that the set
\begin{equation*}
G\coloneqq\{u<-CC_0\}
\end{equation*}
has measure zero for sufficiently large constant $C>0$. We define the function
\begin{equation*}
\varphi=(u+CC_0)_-\in W^{s, 2}(\Omega').
\end{equation*}
Since $\varphi=0$ a.e.\ outside $B_R \setminus \overline{B}_r$, we have from Lemma~\ref{lem-test-V} applied with $\Omega=B_R \setminus \overline{B}_r$ that $\varphi \in V^{s, 2}_0(B_R \setminus \overline{B}_r)$. Thus, we can use $\varphi$ as a test function by Proposition~\ref{prop-test}. Since $\varphi=-u-CC_0>0$ in $G$ and $\varphi=0$ on $G^c$, we have that
\begin{align*}
-C_0 \int_G \varphi \,\mathrm{d}x
&\leq \int_\Omega f\varphi \,\mathrm{d}x \leq \mathcal{E}(u, \varphi) \\
&\leq - \Lambda^{-1}\int_G \int_G \frac{|u(x)-u(y)|^2}{|x-y|^{n+2s}} \,\mathrm{d}y \,\mathrm{d}x + 2\Lambda^{-1}\int_G \int_{G^c} \frac{(u(x)-u(y))\varphi(x)}{|x-y|^{n+2s}} \,\mathrm{d}y \,\mathrm{d}x.
\end{align*}
Since $u(x)-u(y) \leq 0$ for $x \in G$ and $y \in G^c$, we obtain that
\begin{align*}
-C_0 \int_G \varphi \,\mathrm{d}x
&\leq 2\Lambda^{-1}\int_G \int_{\mathbb{R}^n \setminus B_{2R}(x)} \frac{(u(x)-u(y))\varphi(x)}{|x-y|^{n+2s}} \,\mathrm{d}y \,\mathrm{d}x \\
&\leq -2CC_0 \Lambda^{-1} \int_G \varphi(x) \int_{\mathbb{R}^n \setminus B_{2R}(x)} \frac{1}{|x-y|^{n+2s}} \,\mathrm{d}y \,\mathrm{d}x \\
&= - \frac{CC_0 |\mathbb{S}^{n-1}|}{s \Lambda (2R)^{2s}} \int_G \varphi \,\mathrm{d}x.
\end{align*}
Taking $C>s\Lambda(2R)^{2s}/|\mathbb{S}^{n-1}|$ yields $|G|=0$.
\end{proof}

We now prove the localized maximum principle in Theorem~\ref{thm-loc-max}. Corollary~\ref{cor-loc-comp} is a direct consequence of Theorem~\ref{thm-loc-max}.

\begin{proof}[Proof of Theorem~\ref{thm-loc-max}]
Let $\eta \in C^\infty_c(B_{\theta R})$ be such that $0 \leq \eta \leq 1$ in $\mathbb{R}^n$ and $\eta \equiv 1$ on $B_{(\theta+1)R/2}$. Then for any nonnegative function $\varphi \in C^\infty_c(B_R \setminus \overline{B}_r)$, we have
\begin{align*}
\mathcal{E}(u\eta, \varphi)
&\geq - \mathcal{E}(u(1-\eta), \varphi) \\
&= - \int_{B_R \setminus \overline{B}_r} \int_{B_R \setminus \overline{B}_r} (u(x)(1-\eta(x)) - u(y)(1-\eta(y)))(\varphi(x) -\varphi(y)) k(x, y)\,\mathrm{d}y \,\mathrm{d}x \\
&\quad - 2 \int_{B_R \setminus \overline{B}_r} \int_{(B_R \setminus \overline{B}_r)^c} (u(x)(1-\eta(x)) - u(y)(1-\eta(y))) \varphi(x) k(x, y) \,\mathrm{d}y \,\mathrm{d}x \\
&= 2 \int_{B_R \setminus \overline{B}_r} \int_{B_{(\theta+1)R/2}^c} u(y)(1-\eta(y))\varphi(x) k(x, y) \,\mathrm{d}y \,\mathrm{d}x \\
&\geq -2\Lambda \int_{B_R \setminus \overline{B}_r} \varphi(x) \int_{B_{\theta R}^c} \frac{u_-(y)}{|x-y|^{n+2s}} \,\mathrm{d}y \,\mathrm{d}x.
\end{align*}
Since $|x-y| > \frac{\theta-1}{\theta}|y| \geq |y|/2$ for $x \in B_R \setminus \overline{B}_r$ and $y \in B_{\theta R}^c$, we obtain
\begin{equation*}
\mathcal{E}(u\eta, \varphi) \geq -2^{n+2s+1} \Lambda (\theta R)^{-2s} \mathrm{Tail}(u_-; 0, \theta R) \int_{B_R \setminus \overline{B}_r} \varphi \,\mathrm{d}x.
\end{equation*}
Thus, $u\eta$ is a weak supersolution of $\mathcal{L}(u\eta) = -2^{n+2s+1}\Lambda (\theta R)^{-2s} \mathrm{Tail}(u_-; 0, \theta R)$ in $B_R \setminus \overline{B}_r$.

On the other hand, since $u\eta \in W^{s, 2}(\Omega')$ and $u\eta \geq 0$ outside $B_R \setminus \overline{B}_r$, it follows from Lemma~\ref{lem-test-V} that $(u\eta)_- \in V^{s, 2}_0(B_R \setminus \overline{B}_r)$. Therefore, Lemma~\ref{lem-MP-nonhom} shows that
\begin{equation*}
u \geq -C \theta^{-2s} \mathrm{Tail}(u_-; 0, \theta R)
\end{equation*}
a.e.\ in $B_R \setminus \overline{B}_r$, for some constant $C=C(n, s, \Lambda)>0$.
\end{proof}

Finally, we provide the proof of Theorem~\ref{thm-Bocher}.

\begin{proof}[Proof of Theorem~\ref{thm-Bocher}]
We may assume that $u \geq 0$ in $\Omega$. We fix a ball $B_{2R} = B_{2R}(0) \Subset \Omega$ with $R>0$. Theorem~\ref{thm-iso-sing} shows that either $u$ has a removable singularity at the origin, which yields the desired conclusion with $a=0$, or
\begin{equation}\label{eq-asymp-u}
    u(x) \eqsim |x|^{2s-n} \quad\text{in }B_{r}
\end{equation}
for some $r \in (0,R)$. We thus assume for the remainder of the proof that $u$ has a non-removable singularity at the origin and that \eqref{eq-asymp-u} holds.

Recall from Theorem~\ref{thm-FS} that $\Phi \eqsim |x|^{2s-n}$ in a neighborhood of the origin, i.e.,
\begin{equation}\label{eq-asymp-Gamma}
    \Phi(x) \eqsim |x|^{2s-n} \quad\text{in }B_{r}
\end{equation}
for possibly a smaller radius $r$. We take a sufficiently large value $A>0$ such that
\begin{equation}\label{eq-A}
    u+C_0 \mathrm{Tail}(u_-, 0, 2R) < A \left( \frac{1}{2} \Phi - \Phi_R - C_0 \mathrm{Tail}(\Phi, 0, 2R) \right) \quad\text{in }B_{r},
\end{equation}
where $C_0 \coloneqq 2^{-2s}C>0$ for the constant $C>0$ given in Corollary~\ref{cor-loc-comp} and
\begin{equation*}
    \Phi_R \coloneqq \max_{\overline{B}_{2R} \setminus B_R} \Phi.
\end{equation*}
This is possible due to the asymptotic behavior \eqref{eq-asymp-u} and \eqref{eq-asymp-Gamma} and the finiteness of tail terms. Note that the right-hand side of \eqref{eq-A} can be made positive by taking smaller radius $r$ if necessary. Now, we let $a \geq 0$ be the largest value such that 
    \begin{equation}\label{eq-uG}
        u-a\Phi \geq -A \Phi_R -C_0 (\mathrm{Tail}(u_-; 0, 2R) + A\mathrm{Tail}(\Phi; 0, 2R)) \quad \text{in $B_{2R}$}.
    \end{equation}
    The existence of this value is guaranteed since the set of all values $a$ for which \eqref{eq-uG} holds is closed, bounded from above by $A/2$, and contains at least $a=0$.
    
    We define
    \begin{equation*}
        v= u-a\Phi.
    \end{equation*}
    Then, $v$ is $\mathcal{L}$-harmonic in $B_{2R} \setminus \{0\}$ and bounded from below in $B_{2R}$. In view of Theorem~\ref{thm-iso-sing} again, either $v$ has a removable singularity at the origin, or 
    \begin{equation*}
        v(x) \eqsim |x|^{2s-n}
    \end{equation*}
    in a neighborhood of the origin. In the latter case, there exist $r_0 \in (0, r)$ and $\varepsilon \in (0, A/2)$ such that
    \begin{equation*}
        v \geq \varepsilon \Phi - (a+\varepsilon)\Phi_R \quad\text{in }\overline{B}_{r_0}.
    \end{equation*}
    Since $u \geq 0$ in $\Omega$, we also have
    \begin{equation*}
        v \geq \varepsilon \Phi - (a+\varepsilon)\Phi_R \quad\text{in }\overline{B}_{2R} \setminus B_R.
    \end{equation*}
Thus, the localized comparison principle in Corollary~\ref{cor-loc-comp} with $r=r_0$ and $\theta=2$ shows that
    \begin{align*}
        u-(a+\varepsilon)\Phi
        &\geq - (a+\varepsilon)\Phi_R -C_0 \mathrm{Tail}( (u-(a+\varepsilon)\Phi)_-;0, 2R) \\
        &\geq - A \Phi_R -C_0 \mathrm{Tail}( (u-A \Phi)_-;0, 2R) \\
        &\geq - A \Phi_R -C_0 (\mathrm{Tail}(u_-; 0, 2R) + A\mathrm{Tail}(\Phi; 0, 2R))
    \end{align*}
    in $B_{2R}$. It contradicts to the maximality of $a$ in \eqref{eq-uG} and thus $v$ must have a removable singularity at the origin. This completes the proof.
\end{proof}

\section{Liouville theorem for singular solutions}\label{sec-liouville}

In this section, we finally prove the Liouville Theorem~\ref{thm-sing-Liouv-L} for singular solutions to the equation $\mathcal{L}u=0$. The key tools are the localized maximum principle, the isolated singularity theorem, the Harnack inequality on annuli, and the Liouville theorem for globally nonnegative solutions.

We begin with the Harnack inequality on annuli, which follows from the standard Harnack inequality and a covering argument.

\begin{lemma}[Harnack inequality on annulus]\label{lem-harnack}
Let $u$ be a weak solution of $\mathcal{L}u=0$ in $\mathbb{R}^n \setminus B_r$ such that $u \geq 0$ in $\mathbb{R}^n \setminus B_r$. Then for any $R \geq 2r$ and $\theta \geq 2$,
    \begin{equation*}
        \sup_{B_{\theta R} \setminus B_{R}}u \leq C\inf_{B_{\theta R} \setminus B_{R}} u +C\left( \frac{r}{\theta R} \right)^n \fint_{B_r} u_- \,\mathrm{d}x,
    \end{equation*}
    where $C=C(n, s, \Lambda) >0$.    
\end{lemma}

\begin{proof}
    We can cover $B_{\theta R} \setminus B_{R}$ by finitely many open balls $\{B_{\theta R/8}(y_i)\}_{i=1}^N$ with $y_i \in B_{\theta R} \setminus B_{R+\theta R/16}$ so that $B_{\theta R/4}(y_i) \subset \mathbb{R}^n \setminus B_r$. Note that $N$ depends only on $n$. An application of the standard Harnack inequality (see, e.g.\ \cite[Theorem~1.1]{DCKP14}) in $B_{\theta R/8}(y_i)$ shows that
     \begin{equation*}
        \sup_{B_{\theta R/8}(y_i)} u \leq C\inf_{B_{\theta R/8}(y_i)}u +C\mathrm{Tail}(u_-; y_i, \theta R/8),
    \end{equation*}
    where the constant $C>0$ is independent of $i$. Since $u \geq 0$ in $\mathbb{R}^n \setminus B_r$ and 
    \begin{equation*}
        |y-y_i| \geq |y_i|-|y| > \frac{\theta R}{16}+ R-r \geq \frac{\theta R}{16}+\frac{R}{2} \quad \text{for all $y \in B_r$},
    \end{equation*}
    we obtain that
    \begin{equation*}
        \mathrm{Tail}(u_-; y_i, \theta R/8)=\left(\frac{\theta R}{8}\right)^{2s} \int_{B_r}\frac{u_-(y)}{|y-y_i|^{n+2s}}\,\mathrm{d}y \leq  \left(\frac{\theta R}{4}\right)^{2s}\left(\frac{16}{\theta R}\right)^{n+2s} \int_{B_r}u_-(y)\,\mathrm{d}y.
    \end{equation*}
Let $j \in \{1, \ldots, N\}$ be such that
    \begin{equation*}
       \inf_{B_{\theta R} \setminus B_{R}} u \geq \inf_{B_{\theta R/8}(y_j)}u.
    \end{equation*}
Since $n \geq 2$, the annulus $B_{\theta R} \setminus B_{R}$ is connected, and hence
    \begin{equation*}
        \sup_{B_{\theta R} \setminus B_{R}} u \leq \max_{1 \leq i \leq N} \sup_{B_{\theta R/8}(y_i)} u \leq C^N \inf_{B_{\theta R/8}(y_j)} u+C\left( \frac{r}{\theta R} \right)^n \fint_{B_r} u_- \,\mathrm{d}x,
    \end{equation*}
    as desired.
\end{proof}

In the following lemmas, we prove the boundedness of singular solutions in a neighborhood of infinity.

\begin{lemma}\label{lem-bdd}
Assume that $u$ is a weak solution of $\mathcal{L}u=0$ in $\mathbb{R}^n \setminus \overline{B}_{1/2}$ that is bounded from above or below in $\mathbb{R}^n$. Then $u$ is bounded in $\mathbb{R}^n \setminus B_2$.
\end{lemma}

\begin{proof}
We may assume that $u>0$ in $\mathbb{R}^n$. We define
\begin{equation*}
\overline{\Phi} =2\Phi / \min_{\overline{B}_1}\Phi
\end{equation*}
so that $\min_{\overline{B}_1} \overline{\Phi}=2$. Since $\overline{\Phi} \to 0$ as $|x| \to \infty$, there exists $R_0 > 2$ such that
\begin{equation}\label{eq-Gamma-R0}
    \overline{\Phi} \leq 1/4 \quad\text{in }B_{R_0}^c.
\end{equation}
Set for $R > R_0$
\begin{equation*}
    M(R) \coloneqq \sup_{B_{\theta R} \setminus B_{R}} u \quad\text{and}\quad m(R) \coloneqq \inf_{B_{\theta R} \setminus B_{R}} u,
\end{equation*}
where $\theta>2$ to be determined later.

We observe that $u$ and $m(R)(1-\overline{\Phi})$ are $\mathcal{L}$-harmonic in $B_R \setminus \overline{B}_1$, and that
\begin{equation*}
    u \geq m(R)(1-\overline{\Phi}) \quad\text{in }(\overline{B}_{\theta R} \setminus B_{R}) \cup \overline{B}_1.
\end{equation*}
Thus, the localized comparison principle in Corollary~\ref{cor-loc-comp} with $r=1$ shows that
\begin{align*}
    u
    &\geq m(R)(1-\overline{\Phi}) - C \theta^{-2s} (\theta R)^{2s} \int_{B_{\theta R}^c} \frac{(u-m(R)(1-\overline{\Phi}))_-(y)}{|y|^{n+2s}} \,\mathrm{d}y \\
    &\geq m(R)(1-\overline{\Phi}) - C \theta^{-2s} (\theta R)^{2s} \int_{B_{\theta R}^c} \frac{m(R)}{|y|^{n+2s}} \,\mathrm{d}y \\
    &= m(R)\left(1-\overline{\Phi}-\frac{C|\mathbb{S}^{n-1}|}{2s} \theta^{-2s} \right)
\end{align*}
in $B_{R} \setminus \overline{B}_1$.

We take $\theta$ sufficiently large so that $C|\mathbb{S}^{n-1}|\theta^{-2s}/(2s) <1/4$. By using \eqref{eq-Gamma-R0}, we get for $x_0 \in \partial B_{R_0}$ and $R \geq R_0$
\begin{equation}\label{eq-MR0}
    M(R_0) \geq u(x_0) \geq \frac{1}{2}m(R).
\end{equation}
Since Lemma~\ref{lem-harnack} and \eqref{eq-MR0} show that
\begin{equation*}
M(R) \lesssim m(R) \lesssim M(R_0)
\end{equation*}
for any $R\geq R_0$, $M(R)$ is bounded from above for $R \geq R_0$, i.e.\ $u$ is bounded from above in $\mathbb{R}^n \setminus B_{R_0}$. Moreover, $u$ is locally bounded in $\mathbb{R}^n \setminus \overline{B}_1$ by the regularity theory; see, for instance, Theorem~1.1 in Di Castro--Kuusi--Palatucci~\cite{DCKP16}. Therefore, $u$ is bounded in $\mathbb{R}^n \setminus B_r$ for any $r>1$.
\end{proof}

If $u$ is bounded from one side (above or below) near the origin and bounded from the opposite side (below or above, respectively) at infinity, then we require $u$ to be a weak solution in $\mathbb{R}^n \setminus \{0\}$ in order to derive the same result.

\begin{lemma}\label{lem-bdd2}
Assume that $u$ is a weak solution of $\mathcal{L}u=0$ in $\mathbb{R}^n \setminus \{0\}$ that is bounded from one side in a neighborhood of $0$ and that $u$ is bounded from the opposite side in a neighborhood of infinity. Then $u$ is bounded in $\mathbb{R}^n \setminus B_2$.
\end{lemma}

\begin{proof}
We use the same notation as in the proof of Lemma~\ref{lem-bdd}.

We may assume that $u>0$ in $\mathbb{R}^n \setminus B_1$ and $u$ is bounded from above in $B_1$ by assumption and the local boundedness of $u$ in $\mathbb{R}^n \setminus \{0\}$. If $u$ is also bounded from below in $B_1$, then Lemma~\ref{lem-bdd} gives the conclusion. Thus, we assume that $u$ is not bounded from below in $B_1$. Then, Theorem~\ref{thm-iso-sing} and \eqref{eq-FS-bounds} show that there exists a constant $a>0$ such that
\begin{equation*}
    u(x) \geq -a\overline \Phi \quad \text{in $\overline B_1$}.
\end{equation*}
We now prove the desired upper bound of $u$ in $\mathbb{R}^n \setminus B_2$ by contradiction; we suppose that $\limsup_{R \to \infty}M(R)=\infty$. In particular, in view of Lemma~\ref{lem-harnack}, there exists a large radius $\overline{R}$ such that $m(\overline R) >\overline M \coloneqq \max\{2a, 2M(R_0)\}$.

We first observe that
\begin{equation*}
    u \geq m(\overline R)(1-\overline{\Phi}) \quad\text{in }(\overline{B}_{\theta \overline R} \setminus B_{\overline R}) \cup \overline{B}_1.
\end{equation*}
As in the proof of Lemma~\ref{lem-bdd}, we apply the localized comparison principle in Corollary~\ref{cor-loc-comp} to obtain that
\begin{equation*}
M(R_0) \geq \frac{1}{2}m(\overline R),
\end{equation*}
which contradicts the choice of $\overline R$. Therefore, we conclude that $\limsup_{R \to \infty}M(R)<\infty$. Combined with the local boundedness in $\mathbb{R}^n \setminus \{0\}$, this completes the proof.
\end{proof}

\begin{theorem}\label{thm-Liouv-weak}
If $u$ is a nonnegative weak solution of $\mathcal{L}u=0$ in $\mathbb{R}^n$, then $u$ must be constant.
\end{theorem}

\begin{proof}
Lemma~\ref{lem-bdd} shows that $u$ is bounded in $\mathbb{R}^n \setminus B_2$. Moreover, by Theorem~1.1 in Di Castro--Kuusi--Palatucci~\cite{DCKP16}, $u$ is locally bounded in $\mathbb{R}^n$. Therefore, $u \in L^\infty(\mathbb{R}^n)$. Now, it follows from the uniform H\"older estimate (see, e.g.\ Kassmann~\cite[Theorem~1.1]{Kas09}) that
\begin{equation*}
    [u]_{C^\alpha(B_R)} \leq \frac{C}{R^\alpha} \|u\|_{L^\infty(\mathbb{R}^n)},
\end{equation*}
where $\alpha \in (0,1)$ and $C>0$ are constants depending only on $n$, $s$, and $\Lambda$. Passing to the limit $R \to 0$ concludes that $u$ is constant.
\end{proof}

\begin{proof}[Proof of Theorem~\ref{thm-sing-Liouv-L}]
If the singularity at the origin is removable, then $u$ is $\mathcal{L}$-harmonic in $\mathbb{R}^n$. Since $u$ is locally bounded in $\mathbb{R}^n$ and bounded from one side in a neighborhood of infinity, it is bounded from one side in $\mathbb{R}^n$. Therefore, Theorem~\ref{thm-Liouv-weak} concludes that $u$ is constant.

Suppose that the singularity at the origin is not removable. Then it follows from Theorem~\ref{thm-Bocher} that $u=a\Phi+v$ for some constant $a \neq 0$ and an $\mathcal{L}$-harmonic function $v$ in a neighborhood of the origin. Moreover, Lemmas~\ref{lem-bdd} and \ref{lem-bdd2} show that
\begin{equation*}
    b^+\coloneqq \limsup_{|x| \to \infty}u(x) \quad \text{and}\quad b^- \coloneqq \liminf_{|x| \to \infty}u(x)
\end{equation*}
exist as finite values.

Let $\varepsilon>0$. Then there exists $\delta \in (0,1)$ such that
\begin{equation*}
(1-\mathrm{sign}(a)\varepsilon)a \Phi + b^- - \varepsilon \leq u \leq (1+\mathrm{sign}(a)\varepsilon)a \Phi + b^+ + \varepsilon \quad\text{in }\overline{B}_\delta \cup B_{1/\delta}^c.
\end{equation*}
Since $u$ and $(1\pm\mathrm{sign}(a)\varepsilon)a\Phi+b^\pm \pm \varepsilon$ are $\mathcal{L}$-harmonic in $B_{1/\delta} \setminus \overline{B}_\delta$, the standard maximum principle in Theorem~\ref{thm-MP} shows that
\begin{equation*}
(1-\mathrm{sign}(a)\varepsilon)a \Phi + b^- - \varepsilon \leq u \leq (1+\mathrm{sign}(a)\varepsilon)a \Phi + b^+ + \varepsilon \quad\text{in }\mathbb{R}^n.
\end{equation*}
Passing to the limit $\varepsilon \to 0$ yields
\begin{equation*}
    a \Phi + b^- \leq u \leq a \Phi + b^+ \quad\text{in }\mathbb{R}^n.
\end{equation*}
But then Theorem~\ref{thm-iso-sing} shows that $b\coloneqq u-a\Phi$ is $\mathcal{L}$-harmonic in $\mathbb{R}^n$. Therefore, the Liouville Theorem~\ref{thm-Liouv-weak} shows that $b$ is constant.
\end{proof}


\begin{thebibliography}{10}

\bibitem{AdTEJ20}
N.~Alibaud, F.~del Teso, J.~Endal, and E.~R. Jakobsen.
\newblock The {L}iouville theorem and linear operators satisfying the maximum
  principle.
\newblock {\em J. Math. Pures Appl. (9)}, 142:229--242, 2020.

\bibitem{ASS11}
S.~N. Armstrong, B.~Sirakov, and C.~K. Smart.
\newblock Fundamental solutions of homogeneous fully nonlinear elliptic
  equations.
\newblock {\em Comm. Pure Appl. Math.}, 64(6):737--777, 2011.

\bibitem{BGK09}
M.~T. Barlow, A.~Grigor'yan, and T.~Kumagai.
\newblock Heat kernel upper bounds for jump processes and the first exit time.
\newblock {\em J. Reine Angew. Math.}, 626:135--157, 2009.

\bibitem{BL02}
R.~F. Bass and D.~A. Levin.
\newblock Transition probabilities for symmetric jump processes.
\newblock {\em Trans. Amer. Math. Soc.}, 354(7):2933--2953, 2002.

\bibitem{Bat06}
P.~W. Bates.
\newblock On some nonlocal evolution equations arising in materials science.
\newblock In {\em Nonlinear dynamics and evolution equations}, volume~48 of
  {\em Fields Inst. Commun.}, pages 13--52. Amer. Math. Soc., Providence, RI,
  2006.

\bibitem{BC99}
P.~W. Bates and A.~Chmaj.
\newblock A discrete convolution model for phase transitions.
\newblock {\em Arch. Ration. Mech. Anal.}, 150(4):281--305, 1999.

\bibitem{BBK24}
A.~Bj{\"o}rn, J.~Bj{\"o}rn, and M.~Kim.
\newblock Perron solutions and boundary regularity for nonlocal nonlinear
  {D}irichlet problems.
\newblock {\em arXiv preprint arXiv:2406.05994}, 2024.

\bibitem{BGR61}
R.~M. Blumenthal, R.~K. Getoor, and D.~B. Ray.
\newblock On the distribution of first hits for the symmetric stable processes.
\newblock {\em Trans. Amer. Math. Soc.}, 99:540--554, 1961.

\bibitem{Boc03}
M.~B{\^o}cher.
\newblock Singular points of functions which satisfy partial differential
  equations of the elliptic type.
\newblock {\em Bull. Amer. Math. Soc.}, 9(9):455--465, 1903.

\bibitem{BKN02}
K.~Bogdan, T.~Kulczycki, and A.~Nowak.
\newblock Gradient estimates for harmonic and {$q$}-harmonic functions of
  symmetric stable processes.
\newblock {\em Illinois J. Math.}, 46(2):541--556, 2002.

\bibitem{Buc16}
C.~Bucur.
\newblock Some observations on the {G}reen function for the ball in the
  fractional {L}aplace framework.
\newblock {\em Commun. Pure Appl. Anal.}, 15(2):657--699, 2016.

\bibitem{Cac38}
R.~Caccioppoli.
\newblock Sui teoremi d'esistenza di {R}iemann.
\newblock {\em Ann. Scuola Norm. Super. Pisa Cl. Sci. (2)}, 7(2):177--187,
  1938.

\bibitem{CCV11}
L.~Caffarelli, C.~H. Chan, and A.~Vasseur.
\newblock Regularity theory for parabolic nonlinear integral operators.
\newblock {\em J. Amer. Math. Soc.}, 24(3):849--869, 2011.

\bibitem{CV10}
L.~A. Caffarelli and A.~Vasseur.
\newblock Drift diffusion equations with fractional diffusion and the
  quasi-geostrophic equation.
\newblock {\em Ann. of Math. (2)}, 171(3):1903--1930, 2010.

\bibitem{CDL15}
W.~Chen, L.~D'Ambrosio, and Y.~Li.
\newblock Some {L}iouville theorems for the fractional {L}aplacian.
\newblock {\em Nonlinear Anal.}, 121:370--381, 2015.

\bibitem{CK03}
Z.-Q. Chen and T.~Kumagai.
\newblock Heat kernel estimates for stable-like processes on {$d$}-sets.
\newblock {\em Stochastic Process. Appl.}, 108(1):27--62, 2003.

\bibitem{CKW21}
Z.-Q. Chen, T.~Kumagai, and J.~Wang.
\newblock Stability of heat kernel estimates for symmetric non-local
  {D}irichlet forms.
\newblock {\em Mem. Amer. Math. Soc.}, 271(1330):v+89, 2021.

\bibitem{Cim37}
G.~Cimmino.
\newblock Nuovo tipo di condizione al contorno e nuovo metodo di trattazione
  per il problema generalizzato di dirichlet.
\newblock {\em Rendiconti del Circolo Matematico di Palermo (1884-1940)},
  61(1):177--220, 1937.

\bibitem{Cim38}
G.~Cimmino.
\newblock Sulle equazioni lineari alle derivate parziali del secondo ordine di
  tipo ellittico sopra una superficie chiusa.
\newblock {\em Annali della Scuola Normale Superiore di Pisa-Scienze Fisiche e
  Matematiche}, 7(1):73--96, 1938.

\bibitem{Coz17}
M.~Cozzi.
\newblock Regularity results and {H}arnack inequalities for minimizers and
  solutions of nonlocal problems: a unified approach via fractional {D}e
  {G}iorgi classes.
\newblock {\em J. Funct. Anal.}, 272(11):4762--4837, 2017.

\bibitem{DCKP14}
A.~Di~Castro, T.~Kuusi, and G.~Palatucci.
\newblock Nonlocal {H}arnack inequalities.
\newblock {\em J. Funct. Anal.}, 267(6):1807--1836, 2014.

\bibitem{DCKP16}
A.~Di~Castro, T.~Kuusi, and G.~Palatucci.
\newblock Local behavior of fractional {$p$}-minimizers.
\newblock {\em Ann. Inst. H. Poincar\'{e} Anal. Non Lin\'{e}aire},
  33(5):1279--1299, 2016.

\bibitem{DNPV12}
E.~Di~Nezza, G.~Palatucci, and E.~Valdinoci.
\newblock Hitchhiker's guide to the fractional {S}obolev spaces.
\newblock {\em Bull. Sci. Math.}, 136(5):521--573, 2012.

\bibitem{Fal16}
M.~M. Fall.
\newblock Entire {$s$}-harmonic functions are affine.
\newblock {\em Proc. Amer. Math. Soc.}, 144(6):2587--2592, 2016.

\bibitem{FW16}
M.~M. Fall and T.~Weth.
\newblock Liouville theorems for a general class of nonlocal operators.
\newblock {\em Potential Anal.}, 45(1):187--200, 2016.

\bibitem{FK13}
M.~Felsinger and M.~Kassmann.
\newblock Local regularity for parabolic nonlocal operators.
\newblock {\em Comm. Partial Differential Equations}, 38(9):1539--1573, 2013.

\bibitem{FRRO24}
X.~Fern\'andez-Real and X.~Ros-Oton.
\newblock {\em Integro-differential elliptic equations}, volume 350 of {\em
  Progress in Mathematics}.
\newblock Birkh\"auser/Springer, Cham, [2024] \copyright 2024.

\bibitem{FSV15}
A.~Fiscella, R.~Servadei, and E.~Valdinoci.
\newblock Density properties for fractional {S}obolev spaces.
\newblock {\em Ann. Acad. Sci. Fenn. Math.}, 40(1):235--253, 2015.

\bibitem{Gar19}
N.~Garofalo.
\newblock Fractional thoughts.
\newblock In {\em New developments in the analysis of nonlocal operators},
  volume 723 of {\em Contemp. Math.}, pages 1--135. Amer. Math. Soc.,
  [Providence], RI, [2019] \copyright 2019.

\bibitem{GS56}
D.~Gilbarg and J.~Serrin.
\newblock On isolated singularities of solutions of second order elliptic
  differential equations.
\newblock {\em J. Analyse Math.}, 4:309--340, 1955/56.

\bibitem{GO08}
G.~Gilboa and S.~Osher.
\newblock Nonlocal operators with applications to image processing.
\newblock {\em Multiscale Model. Simul.}, 7(3):1005--1028, 2008.

\bibitem{GHH17}
A.~Grigor'yan, E.~Hu, and J.~Hu.
\newblock Lower estimates of heat kernels for non-local {D}irichlet forms on
  metric measure spaces.
\newblock {\em J. Funct. Anal.}, 272(8):3311--3346, 2017.

\bibitem{GHH18}
A.~Grigor'yan, E.~Hu, and J.~Hu.
\newblock Two-sided estimates of heat kernels of jump type {D}irichlet forms.
\newblock {\em Adv. Math.}, 330:433--515, 2018.

\bibitem{GHL14}
A.~Grigor'yan, J.~Hu, and K.-S. Lau.
\newblock Estimates of heat kernels for non-local regular {D}irichlet forms.
\newblock {\em Trans. Amer. Math. Soc.}, 366(12):6397--6441, 2014.

\bibitem{GK25}
T.~Grzywny and M.~Kwa\'snicki.
\newblock Liouville's theorems for {L}\'evy operators.
\newblock {\em Math. Ann.}, 391(4):5857--5910, 2025.

\bibitem{Kas09}
M.~Kassmann.
\newblock A priori estimates for integro-differential operators with measurable
  kernels.
\newblock {\em Calc. Var. Partial Differential Equations}, 34(1):1--21, 2009.

\bibitem{Kas11}
M.~Kassmann.
\newblock A new formulation of {H}arnack's inequality for nonlocal operators.
\newblock {\em C. R. Math. Acad. Sci. Paris}, 349(11-12):637--640, 2011.

\bibitem{KKL23}
M.~Kassmann, M.~Kim, and K.-A. Lee.
\newblock Robust near-diagonal {G}reen function estimates.
\newblock {\em Int. Math. Res. Not. IMRN}, (19):16957--16993, 2023.

\bibitem{KW24b}
M.~Kassmann and M.~Weidner.
\newblock The parabolic {H}arnack inequality for nonlocal equations.
\newblock {\em Duke Math. J.}, 173(17):3413--3451, 2024.

\bibitem{Kel26}
O.~D. Kellogg.
\newblock On some theorems of {B}\^ocher concerning isolated singular points of
  harmonic functions.
\newblock {\em Bull. Amer. Math. Soc.}, 32(6):664--668, 1926.

\bibitem{KV86}
S.~Kichenassamy and L.~V\'{e}ron.
\newblock Singular solutions of the {$p$}-{L}aplace equation.
\newblock {\em Math. Ann.}, 275(4):599--615, 1986.

\bibitem{Kim25}
M.~Kim.
\newblock Removable singularities for nonlocal minimal graphs.
\newblock {\em arXiv preprint arXiv:2501.14299}, 2025.

\bibitem{KL23}
M.~Kim and S.-C. Lee.
\newblock Supersolutions and superharmonic functions for nonlocal operators with {O}rlicz growth.
\newblock {\em arXiv preprint arXiv:2311.01246}, 2023.

\bibitem{KL24}
M.~Kim and S.-C. Lee.
\newblock Singularities of solutions of nonlocal nonlinear equations.
\newblock {\em arXiv preprint arXiv:2410.13292}, 2024.

\bibitem{KW24a}
M.~Kim and M.~Weidner.
\newblock Optimal boundary regularity and {G}reen function estimates for
  nonlocal equations in divergence form.
\newblock {\em arXiv preprint arXiv:2408.12987}, 2024.

\bibitem{Kli25}
T.~Klimsiak.
\newblock B{\^o}cher type theorem for elliptic equations with drift perturbed
  {L}\'evy operator.
\newblock {\em arXiv preprint arXiv:2503.03354}, 2025.

\bibitem{KKP17}
J.~Korvenp\"{a}\"{a}, T.~Kuusi, and G.~Palatucci.
\newblock Fractional superharmonic functions and the {P}erron method for
  nonlinear integro-differential equations.
\newblock {\em Math. Ann.}, 369(3-4):1443--1489, 2017.

\bibitem{Kur06}
M.~Kurzke.
\newblock A nonlocal singular perturbation problem with periodic well
  potential.
\newblock {\em ESAIM Control Optim. Calc. Var.}, 12(1):52--63, 2006.

\bibitem{Lab00}
D.~A. Labutin.
\newblock Removable singularities for fully nonlinear elliptic equations.
\newblock {\em Arch. Ration. Mech. Anal.}, 155(3):201--214, 2000.

\bibitem{Lab01}
D.~A. Labutin.
\newblock Isolated singularities for fully nonlinear elliptic equations.
\newblock {\em J. Differential Equations}, 177(1):49--76, 2001.

\bibitem{LLWX20}
C.~Li, C.~Liu, Z.~Wu, and H.~Xu.
\newblock Non-negative solutions to fractional {L}aplace equations with
  isolated singularity.
\newblock {\em Adv. Math.}, 373:107329, 38, 2020.

\bibitem{LWX18}
C.~Li, Z.~Wu, and H.~Xu.
\newblock Maximum principles and {B}\^ocher type theorems.
\newblock {\em Proc. Natl. Acad. Sci. USA}, 115(27):6976--6979, 2018.

\bibitem{LW25}
N.~Liao and M.~Weidner.
\newblock Nonnegative solutions to nonlocal parabolic equations.
\newblock {\em arXiv preprint arXiv:2505.08449}, 2025.

\bibitem{LL14}
E.~Lindgren and P.~Lindqvist.
\newblock Fractional eigenvalues.
\newblock {\em Calc. Var. Partial Differential Equations}, 49(1-2):795--826,
  2014.

\bibitem{LSW63}
W.~Littman, G.~Stampacchia, and H.~F. Weinberger.
\newblock Regular points for elliptic equations with discontinuous
  coefficients.
\newblock {\em Ann. Scuola Norm. Sup. Pisa Cl. Sci. (3)}, 17:43--77, 1963.

\bibitem{NSS03}
F.~Nicolosi, I.~V. Skrypnik, and I.~I. Skrypnik.
\newblock Precise point-wise growth conditions for removable isolated
  singularities.
\newblock {\em Comm. Partial Differential Equations}, 28(3-4):677--696, 2003.

\bibitem{Pic24}
E.~Picard.
\newblock Quelques th\'eor\`emes \'el\'ementaires sur les fonctions
  harmoniques.
\newblock {\em Bull. Soc. Math. France}, 52:162--166, 1924.

\bibitem{Ray26}
G.~E. Raynor.
\newblock Isolated singular points of harmonic functions.
\newblock {\em Bull. Amer. Math. Soc.}, 32(5):537--544, 1926.

\bibitem{Rie49}
M.~Riesz.
\newblock L'int\'egrale de {R}iemann-{L}iouville et le probl\`eme de {C}auchy.
\newblock {\em Acta Math.}, 81:1--223, 1949.

\bibitem{SKM93}
S.~G. Samko, A.~A. Kilbas, and O.~I. Marichev.
\newblock {\em Fractional integrals and derivatives}.
\newblock Gordon and Breach Science Publishers, Yverdon, 1993.
\newblock Theory and applications, Edited and with a foreword by S. M.
  Nikol'ski\u i, Translated from the 1987 Russian original, Revised by the
  authors.

\bibitem{Ser64a}
J.~Serrin.
\newblock Local behavior of solutions of quasi-linear equations.
\newblock {\em Acta Math.}, 111:247--302, 1964.

\bibitem{Ser64b}
J.~Serrin.
\newblock Pathological solutions of elliptic differential equations.
\newblock {\em Ann. Scuola Norm. Sup. Pisa Cl. Sci. (3)}, 18:385--387, 1964.

\bibitem{Ser65a}
J.~Serrin.
\newblock Isolated singularities of solutions of quasi-linear equations.
\newblock {\em Acta Math.}, 113:219--240, 1965.

\bibitem{Ser65b}
J.~Serrin.
\newblock Removable singularities of solutions of elliptic equations. {II}.
\newblock {\em Arch. Rational Mech. Anal.}, 20:163--169, 1965.

\bibitem{Sil07}
L.~Silvestre.
\newblock Regularity of the obstacle problem for a fractional power of the
  {L}aplace operator.
\newblock {\em Comm. Pure Appl. Math.}, 60(1):67--112, 2007.

\bibitem{Sil16}
L.~Silvestre.
\newblock A new regularization mechanism for the {B}oltzmann equation without
  cut-off.
\newblock {\em Comm. Math. Phys.}, 348(1):69--100, 2016.

\bibitem{Ver96}
L.~V\'eron.
\newblock {\em Singularities of solutions of second order quasilinear
  equations}, volume 353 of {\em Pitman Research Notes in Mathematics Series}.
\newblock Longman, Harlow, 1996.

\bibitem{Wey40}
H.~Weyl.
\newblock The method of orthogonal projection in potential theory.
\newblock {\em Duke Math. J.}, 7:411--444, 1940.

\end{thebibliography}

\end{document}